\documentclass[11pt]{article}
\usepackage[section]{placeins}
\usepackage{multirow}
\usepackage{geometry}
\usepackage{amssymb}
\usepackage{amsmath}
\usepackage{float}
\usepackage{color}
\usepackage{graphicx}
\usepackage{subfigure}
\usepackage{tabu}
\usepackage{booktabs}
\usepackage{array}
\usepackage{caption}
\usepackage{amsthm}
\usepackage{booktabs}
\usepackage{cite}
\usepackage{indentfirst}
\numberwithin{equation}{section}
\newtheorem{theorem}{\bf Theorem}[section]
\usepackage{algpseudocode}
\usepackage{algorithm}
\usepackage{algorithmicx}

\usepackage{booktabs}
\usepackage{hyperref}
\usepackage{arydshln}
\allowdisplaybreaks[4]
\makeatletter
\@addtoreset{equation}{section}
\makeatother

\title{
	Randomized GCUR decompositions
}
\author{
	Zhengbang Cao\footnote{School of Mathematical Sciences, Ocean University of China, Qingdao 266100, China.
		E-Mail: {\tt caozhengbang@stu.ouc.edu.cn}}
	\and Yimin Wei\footnote{School of Mathematical Sciences and and Key Laboratory of Mathematics for Nonlinear Sciences, Fudan University, Shanghai 200433, China.
		E-Mail: {\tt ymwei@fudan.edu.cn}}
	\and	Pengpeng Xie\footnote{Corresponding author (P. Xie). School of Mathematical Sciences, Ocean University of China, Qingdao 266100, China.
		E-Mail: {\tt xie@ouc.edu.cn}.
	}
}
\date{}

\begin{document}
	\maketitle
	\begin{abstract}
		By exploiting the random sampling techniques, this paper derives an efficient randomized algorithm for computing a generalized CUR decomposition, which provides low-rank approximations
		of both matrices simultaneously in terms of some of their rows and columns.
		For	large-scale data sets that are expensive to store and manipulate, a new variant of the discrete empirical interpolation method known as L-DEIM, which needs much lower cost and provides a significant acceleration in practice, is also combined with the random sampling approach to further improve the efficiency of our algorithm.
		Moreover, adopting the randomized algorithm to implement the truncation process of restricted singular value decomposition (RSVD), combined with the L-DEIM procedure, we propose a fast algorithm for computing an RSVD based CUR decomposition, which provides a coordinated low-rank approximation of the three matrices in a CUR-type format simultaneously and provides advantages over the standard CUR approximation for some applications. We establish detailed probabilistic error analysis for the algorithms and provide numerical results that show the promise of our approaches.
		\\ \hspace*{\fill} 
		\\ \hspace*{\fill} 
	\end{abstract}
	\section{Introduction}
	\hskip 2em
	The ability to extract meaningful insights from large and complex datasets is a key challenge in data analysis,
	and most efforts have been focused on manipulating, understanding and interpreting large-scale data matrices.
	In many cases, matrix factorization
	methods are employed for constructing parsimonious and informative representations to facilitate computation and interpretation.
	A principal approach is the CUR decomposition \cite{drineas2008SIAMrelative,sorensen2016SIAMdeim,mahoney2009PNAScur,wang2013JMLRimproving},
	which is a low-rank approximation of a matrix $A\in\mathbb{R}^{m\times n}$ of the form
	\begin{equation}\label{CUR decomposition}
		A \approx CUR,
	\end{equation}
	where matrices $C\in\mathbb{R}^{m\times k_1}$ and $R\in\mathbb{R}^{k_2\times n}$ are subsets of the columns and rows, respectively, of the original matrix $A$.
	The $k_1 \times k_2$ matrix $U$ is constructed to guarantee that CUR is a well-approximated decomposition.
	As described in \cite{sorensen2016SIAMdeim}, the CUR decomposition is a powerful technique for handling large-scale data sets,
	providing two key advantages over the singular value decomposition (SVD)
	$A\approx VSW^{\mathrm{T}}$
	: first, when $A$ is sparse, so too are $C$ and $R$, unlike the dense matrices $V$ and $W$ of singular vectors;
	second, the matrices $C$
	and $R$ contains actual columns and rows of the original matrix, preserving certain features such as sparse, nonnegative, integer-valued and so on.
	
	\hskip 2em
	These novel properties of CUR are desirable for feature
	selection and data interpretation, leading to extensive research and also making CUR attractive in a wide range of applications \cite{boutsidis2014optimal,wang2013JMLRimproving,hamm2021SIAMperturbations,cai2021SIAMrobust,cai2020IEEErapid,cai2022Arixvmatrix}.
	
	Recently, in \cite{gidisu2022SIAMgeneralized}, Gidisu and Hochstenbach developed a generalized CUR decomposition (GCUR) for matrix pair $A$ and $B$ with the same number of columns:
	$A$ is $m \times n$, $B$ is $d \times n$ and both are of full column rank, which can be viewed as a CUR decomposition of $A$ relative to $B$.
	{This novel decomposition is suitable for applications where the goal is to extract the most distinctive information from a particular data set compared to another.}
	Furthermore, it is well known that canonical correlation analysis (CCA) \cite{xu2013ARXIVsurvey} is an extremely useful
	statistical method to analyze the correlation between two sets of variables,  which has been applied in a variety of fields, including economics, psychology, and biology, among others.
	Motivated by CCA, in \cite{gidisu2022ARXIVrsvd}, Gidisu and Hochstenbach developed a new coordinated CUR factorization of a matrix
	triplet $(A, B, G)$ of compatible dimensions, based on the restricted singular value decomposition (RSVD) \cite{zha1991SIAMrestricted}.
	This factorization was called an RSVD based CUR (RSVD-CUR) factorization.
	The RSVD-CUR factorization serves as a valuable instrument for facilitating multi-view/label dimensionality reduction, as well as addressing a specific class of perturbation problems.
	Typically, the identification of the optimal subset of rows and columns to generate a CUR-type factorization involves several methods.
	Two sampling techniques employed in \cite{gidisu2022ARXIVrsvd,gidisu2022SIAMgeneralized} are named DEIM \cite{barrault2004CRMempirical,chaturantabut2010SIAMnonlinear} and L-DEIM \cite{gidisu2022Arxivhybrid}.
	Specifically, as the inputs, the DEIM and L-DEIM require the generalized SVD (GSVD) of the matrix pair $(A,B)$ and the RSVD of the matrix triplet $(A, B, G)$ for sampling when constructing the GCUR and RSVD-CUR decomposition, respectively.
	The overall computational complexity of the algorithms discussed in \cite{gidisu2022ARXIVrsvd,gidisu2022SIAMgeneralized} is dominated by the construction of the GSVD and the RSVD.
	However, in practice, this cost can	be prohibitively expensive, making it unsuitable for large-scale applications.
	
	\hskip 2em
	The utility of randomized algorithms in facilitating the process of matrix decomposition has been well-established.
	Such algorithms serve to reduce both the computational complexity of deterministic methods and the communication among different levels of memories that can severely impede the performance of contemporary computing architectures.
	Based on the framework in \cite{halko2011SIAMfinding}, many computationally efficient methods for implementing large-scale matrix factorizations have been proposed, analyzed, and implemented, such as \cite{wei2016SIAMtikhonov,wei2021CAMCrandomized,saibaba2021NLAArandomized,saibaba2016NLAArandomized}.
	Meanwhile, these well-established randomized algorithms have been widely used for many
	practical applications, such as the least squares problems \cite{xie2019NLAArandomized,boutsidis2009LAArandom,zhang2020JCAMrandomized} and Tikhonov regularization \cite{xiang2013IPregularization,rachkovskij2012CSArandomized}.
	Motivated by this success, in this work we introduce the randomized
	schemes for efficiently computing the GCUR and the RSVD-CUR decomposition.
	To be specific, there are two main computational stages involved in our randomized algorithms.
	In the first stage, we use random projections to identify a subspace that captures most of the action of the input matrix.
	Then we project the input matrix onto this subspace and get a reduced matrix which is then manipulated deterministically to obtain the desired low-rank approximation of the GSVD and RSVD.
	The	second stage can be accomplished with well-established deterministic methods DEIM and L-DEIM, operating on the approximation obtained in the first stage to sample the columns and rows of the original matrices.
	Compared with non-random approaches, our algorithms allow for comparable accuracy with much lower cost and will be more computationally efficient on large-scale data.
	Details of the algorithm, theoretical analysis and numerical results are provided to show the effectiveness of our approaches.
	
	\hskip 2em
	The rest of this paper is organized as follows.
	In Section 2, we first give a concise overview of the GSVD and the RSVD,
	then we introduce some basic notation and describe several sampling techniques including the DEIM and L-DEIM.
	Next, in Section 3, we present our randomized algorithms for computing the GCUR factorization using the DEIM and L-DEIM procedures, where the probabilistic error bound is also presented in detail.
	In Section 4, we first briefly review the literature on existing algorithms for the computation of the RSVD, and develop an efficient method for computing this decomposition.
	Then we develop randomized algorithms for computing the RSVD-CUR decomposition based on the sampling procedure L-DEIM, along with detailed probabilistic error analysis.
	In Section 5, we test the performance of the proposed algorithms on several synthetic matrices and real-world datasets.
	Finally, in Section 6, we end this paper with concluding remarks.

	\section{Preliminaries}
	\hskip 2em
	In this paper, we adopt the MATLAB notation for indexing matrices and vectors.
	Specifically, we use $X(\mathbf{q},:)$ to represent the $k$ rows of $X$, which are determined by the values of the vector $\mathbf{q}\in\mathbb{N}^k_+$,
	and $X(:,\mathbf{p})$ to denote the $k$ columns of $X$ that correspond to the indices in $\mathbf{p}$.
	Additionally, we denote the 2-norm of matrices and vectors by $\| \cdot \|$,
	and utilize $A^{\dagger}$ to represent the Moore-Penrose pseudoinverse \cite{wei2018WSnumerical} of $A$.


	\subsection{GSVD and RSVD}
	\hskip 2em
	We provide a concise review of the GSVD and RSVD which constitute the fundamental components of the proposed algorithms.
	To be consistent with \cite{gidisu2022SIAMgeneralized}, this paper employs the formulation of GSVD proposed by Van Loan in \cite{van1985NMcomputing},
	while additional formulations and contributions to the GSVD can be found in
	\cite{van1976SIAMgeneralizing, paige1981SIAMtowards,huang2022JSC,stewart1982NMcomputing,sun1983SIAMperturbation,bai1993SIAMcomputing,zha1996NMcomputing}.
	Let
	$A\in\mathbb{R}^{m \times n}$ and $B\in\mathbb{R}^{d \times n}$
	with both $ m\ge n$ and $d \ge n$, then there exist orthogonal matrices
	$U\in\mathbb{R}^{m \times m}$, $V\in\mathbb{R}^{d\times d}$
	and a nonsingular $Y\in\mathbb{R}^{n \times n}$ such that
	\begin{equation}\label{GSVD of B}
		 B = V \Sigma Y^{\mathrm{T}},\
		 \Sigma=\operatorname{diag}(\beta_1,\ldots,\beta_n),\  \beta_i\in\left[0,1\right],
	\end{equation}
	\begin{equation}\label{GSVD of A}
		A = U \Gamma Y^{\mathrm{T}},\
		\Gamma=\operatorname{diag}(\gamma_1,\ldots,\gamma_n),\  \gamma_i\in\left[0,1\right],
	\end{equation}
	where $\gamma_i^2+\beta_i^2=1$ and the ratios $\gamma_i/\beta_i$ are in a non-increasing order for $i=1,\ldots,n.$
	Further, nonnegative number pairs $\{\gamma_i,\beta_i\}_{i=1}^n$ are actually the generalized singular values of the matrix pair $(A,B)$ as defined in \cite{sun1983SIAMperturbation}, and
	the sensitivity of the generalized singular values of a matrix pair to
	perturbations in the matrix elements was analyzed in \cite{sun1982MNSperturbation,li1993SIAMbounds,sun1983SIAMperturbation}.
	
	\hskip 2em
	The RSVD \cite{de1991SIAMrestricted,zha1991SIAMrestricted} is the factorization of a given
	matrix, relative to two other given matrices, which can be interpreted as the ordinary singular value decomposition 
    with different inner products in the row and column spaces.
	Consider the matrix triplet
	$A\in\mathbb{R}^{m\times n}$, $B\in\mathbb{R}^{m\times l}$ and $G\in\mathbb{R}^{d\times n}$, where $\ell \ge d \ge m \ge n$.
	We assume that $B$ and $G$ are of full rank.
	Following the formulation of the RSVD proposed by Zha \cite{zha1991SIAMrestricted}, there exist
	orthogonal matrices $U\in\mathbb{R}^{l\times l}$, $V\in\mathbb{R}^{d\times d}$
	and nonsingular matrices $Z\in\mathbb{R}^{m\times m}$ and $W\in\mathbb{R}^{n\times n}$ such that
	\begin{equation}\label{RSVD of (A,B,G)}
		A=ZD_AW^{\mathrm{T}},\  B=ZD_BU^{\mathrm{T}},\  G=VD_GW^{\mathrm{T}},
	\end{equation}
	or alternatively it can be expressed conveniently as
	\begin{equation*}
		\left[\begin{array}{ll}
			A & B \\
			G &
		\end{array}\right]=\left[\begin{array}{ll}
			Z & \\
			& V
		\end{array}\right]\left[\begin{array}{ll}
			D_A & D_B \\
			D_G &
		\end{array}\right]\left[\begin{array}{ll}
			W & \\
			& U
		\end{array}\right]^{\mathrm{T}},
	\end{equation*}
	where $D_A\in\mathbb{R}^{m\times n}$, $D_B\in\mathbb{R}^{m\times l}$, and $D_G\in\mathbb{R}^{d\times n}$ are nonnegative diagonal matrices.
	\subsection{Subset Selection Procedure}\label{subsection: Subset selection procedure}
	\hskip 2em
	In this section, we describe three different techniques for extracting appropriate subsets of columns or rows from matrices, that are the deterministic leverage score sampling procedure, the DEIM algorithm and the L-DEIM algorithm.
	
	\hskip 2em
	Given $A\in \mathbb{R}^{m \times n}$ with $\mathrm{rank}(A)\ge k$.
	Let $V_k$ contain its $k$ leading right singular vectors, and we denote the $i$th row of $V_k$ by $\left[V_k\right]_{i,:}$.
	Then the rank-$k$ leverage score of the $i$th column of $A$ is defined as
	\begin{equation*}
		\ell_i^k=\left\|\left[V_k\right]_{i,:}\right\|^2, \quad i=1, \ldots, n.
	\end{equation*}
	The deterministic leverage score sampling procedure  \cite{jolliffe1972JRSSdiscarding,papailiopoulos2014ACMprovable}
	selects the columns of $A$ that correspond to the largest leverage scores $\ell_i^k$, for a given $k\le \mathrm{rank}(A)$.
	From a practical perspective, this deterministic procedure 
	is notably straightforward to implement and computationally efficient,
	but 
	it lacks the capability to provide rigorous performance guarantees.

	\hskip 2em
	The DEIM selection algorithm was initially introduced in \cite{chaturantabut2010SIAMnonlinear}
	as a technique for model order reduction in nonlinear dynamical systems,
	and was employed to produce a two-side interpolatory decomposition in \cite{voronin2014ARxivcur}.
	In order to better understand the DEIM algorithm, it is necessary to first introduce the interpolatory projectors.
	Given a full column rank basis matrix $V\in \mathbb{R}^{m \times k}$ with $k\le m$ and a set of distinct indices $\mathbf{p}$, the interpolatory projector for $\mathbf{p}$ onto the range of $V$ $\mathrm{Ran}(V)$ is
	\begin{equation*}
		\mathbb{P}=V(P^{\mathrm{T}}V)^{-1}P^{\mathrm{T}},
	\end{equation*}
	where $P=I(:,\mathbf{p})\in \mathbb{R}^{m \times k}$, provided $P^{\mathrm{T}}V$ is invertible.
	The oblique projector $\mathbb{P}$ has an important property:
	for any vector $x\in \mathbb{R}^{m}$,
	\begin{equation}
		(\mathbb{P} x)(\mathbf{p})
		=P^{\mathrm{T}} \mathbb{P} x
		=P^{\mathrm{T}} V \left(P^{\mathrm{T}} V\right)^{-1} P^{\mathrm{T}} x
		=P^{\mathrm{T}} x
		=x(\mathbf{p}),
	\end{equation}
	hence the projected vector $Px$ matches $x$ in the entries corresponding to the indices in $\mathbf{p}$.
	The DEIM procedure processes the columns of $V$ sequentially, starting from the most significant singular vector to the least significant one.
	The first index corresponds to the largest magnitude entry in the first dominant singular vector.
	The next index corresponds to the largest entry in the subsequent singular vectors, after the interpolatory projection in the previous direction has been removed.
	See Algorithm \ref{Al-DEIM} for details.
	\begin{algorithm}[htb]
		\caption{DEIM index selection \cite{chaturantabut2010SIAMnonlinear} }
		\label{Al-DEIM}
		\hspace*{0.02in} {\bf Input:}
		$V \in \mathbb{R}^{m \times k}$ with $k \leq \mathrm{min}(m,n)$.
		\\
		\hspace*{0.02in} {\bf Output:}
		column index $\mathbf{p}\in\mathbb{N}^k_{+}$, with non-repeating entries,
		$V \in \mathbb{R}^{m \times k}$ with $k \leq \mathrm{min}(m,n)$.
		\begin{algorithmic}[1]
			\State  $v=V(:, 1)$.
			\State $p_{1}=\operatorname{argmax}_{1 \leq i \leq m}\left|v_{i}\right|$.
			\State $\mathbf{p}=[\begin{array}{l}
				p_1
			\end{array}]$.
			\For{$j=2, \ldots, k$}
			\State $v=V(:, j)$.
			\State $r=v-V(:, 1: j-1) (V(\mathbf{p}, 1: j-1)\backslash v(\mathbf{p}))$.
			\State $p_{j}=\operatorname{argmax}_{1 \leq i \leq m}\left|{r}_{i}\right|$.
			\State $\mathbf{p}=\left[\begin{array}{ll}\mathbf{p} & p_{j}\end{array}\right]$.
			\EndFor
		\end{algorithmic}
	\end{algorithm}
	
	\hskip 2em
	In \cite{sorensen2016SIAMdeim}, Sorensen and Embree adapted the DEIM in the context of subset selection to matrix CUR factorization, illustrating its favorable performance outperforms the leverage scores method.
	However, when used for index selection, the DEIM algorithm has an evident drawback:
	the quantity of singular vectors available sets a limit on the number of indices that can be selected.
	
	\hskip 2em
	This problem can be effectively resolved by employing the L-DEIM algorithm (Algorithm \ref{Al-LDEIM}), which was proposed by Gidisu and Hochstenbach \cite{gidisu2022Arxivhybrid}.
	This novel variation of the DEIM algorithm is a fusion of DEIM and the leverage score sampling method, resulting in an approach that offers enhanced efficiency and comparable accuracy.
	Specifically, while constructing a rank-${k}$ CUR decomposition, the first $\widehat{k}$ indices are obtained from the standard DEIM procedure, where $\widehat{k}<k$.
	Then we calculate the leverage scores of the residual singular vector, sorted in descending order, and output the additional $k-\widehat{k}$ indices that correspond to the largest $k-\widehat{k}$ leverage scores.
	According to the conclusion summarized in \cite{gidisu2022Arxivhybrid}, the accuracy of the L-DEIM procedure can be comparable with the DEIM when the target rank $k$ is at most twice the available $\widehat{k}$ singular vectors, and empirically, we can set $\widehat{k}=k/2$.
	
	
	\begin{algorithm}[htb]
		\caption{L-DEIM index selection \cite{gidisu2022Arxivhybrid} }
		\label{Al-LDEIM}
		\hspace*{0.02in} {\bf Input:}
		$V \in \mathbb{R}^{m \times \widehat{k}}$, target rank $k$ with $\widehat{k} \leq k \leq \min (m, n)$.
		\\
		\hspace*{0.02in} {\bf Output:}
		column indices $\mathbf{p}\in\mathbb{N}^k_{+}$ with non-repeating entries.
		\begin{algorithmic}[1]
			\For{$j=1,\ldots,\widehat{k}$}
			\State $\mathbf{p}(j)=\operatorname{argmax}_{1 \leq i \leq m}\left|(V(:, j))_{i}\right|$.
			\State $V(:, j+1)=V(:, j+1)-V(:, 1: j) \cdot(V(\mathbf{p}, 1: j) \backslash V(\mathbf{p}, j+1))$.
			\EndFor
			\State Compute $\ell_{i}=\left\|V_{i:}\right\|^2 \quad$ for $i=1, \ldots,m$.
			\State Sort $\ell$ in non-increasing order.
			\State Remove entries in $\ell$ corresponding to the indices in $\mathbf{p}$.
			\State $\mathbf{p}^{\prime}=k-\widehat{k}$ indices corresponding to $k-\widehat{k}$ largest entries of $\ell$.
			\State $\mathbf{p}=\left[\mathbf{p} ; \mathbf{p}^{\prime}\right]$.
		\end{algorithmic}
	\end{algorithm}
	\section{Randomization for GCUR}
	\hskip 2em
	In this section, we first give a brief introduction to the GCUR factorization.
	Moreover, by combining the random sampling techniques with the DEIM and L-DEIM procedures, we establish two versions of efficient randomized algorithms for computing this factorization, along with the detailed probabilistic error analysis for our approaches.
	
	\subsection{GCUR}
	\hskip 2em
	In \cite{gidisu2022SIAMgeneralized}, Gidisu and Hochstenbach developed a novel matrix factorization called GCUR decomposition, which applies to a pair of matrices $A$ and $B$ with the equal number of columns.
	The authors explain that this factorization can be regarded as a CUR decomposition of $A$ relative to $B$. 
	Given a matrix pair $(A,B)$, where $A$ is $m \times n$ and $B$ is $d \times n$ and both are of full column ranks with $m\ge n$ and $d\ge n$, then the rank-$k$ GCUR decomposition of $(A,B)$ is a matrix approximation of $A$ and $B$ expressed as
	\begin{equation}\label{GCUR of A}
		A\approx C_A M_A R_A
		=A(:,\mathbf{p})\ M_A\ A(\mathbf{s}_A,:),
	\end{equation}
	\begin{equation}\label{GCUR of B}
		B\approx C_B M_B R_B
		=B(:,\mathbf{p})\ M_B\ B(\mathbf{s}_B,:).
	\end{equation}
	Here matrices $C_A$ and $C_B$ indexed by the vector $\mathbf{p}$ are the subsets of the columns of $A$ and $B$, capturing the most relevant information of the original matrix.
	Meanwhile, $R_A$ and $R_B$ are formed by extracting $k$
	rows from $A$ and $B$, where the selected row
	indices are stored in the vectors $\mathbf{s}_A$ and $\mathbf{s}_B$, respectively.
	Once the row/column indices have been specified, one can choose different ways to construct the middle matrices $M_A$ and $M_B$.
	Following the work in \cite{sorensen2016SIAMdeim,mahoney2009PNAScur,stewart1999NMfour}, the authors in \cite{gidisu2022SIAMgeneralized} construct $M_A$ and $M_B$ as
	\begin{equation*}
		M_A=C_A^{\dagger} A R_A^{\dagger}
		   =(C_A^{\mathrm{T}}C_A)^{-1}C^{\mathrm{T}}_A A R_A^{\mathrm{T}}(R_AR_A^{\mathrm{T}})^{-1},
	\end{equation*}
	\begin{equation*}
		M_B=C_B^{\dagger} B R_B^{\dagger}
		   =(C_B^{\mathrm{T}}C_B)^{-1}C^{\mathrm{T}}_B B R_B^{\mathrm{T}}(R_B R_B^{\mathrm{T}})^{-1},
	\end{equation*}
	yielding the GCUR decomposition that can be realized by the following steps:
	first, the columns of $A$ are projected onto the range of $C_A$;
	then projecting it onto the row space of $R_A$. 
	
	\hskip 2em
	In essence, this factorization is a generalization of the CUR decomposition, and it has a close connection with the CUR decomposition of $AB^{\dagger}$.
	A detailed discussion of the properties can be found in \cite[Proposition 4.2]{gidisu2022SIAMgeneralized}.
	In \cite{gidisu2022SIAMgeneralized}, the choice for the sampling indices is guided by knowledge of the GSVD.
	Specifically, given the GSVD for matrix pair of the form (\ref{GSVD of A}) and (\ref{GSVD of B}), the DEIM procedure uses $U$, $V$ and $Y$ to select the indices $\mathbf{s}_A$, $\mathbf{s}_B$ and $\mathbf{p}$ respectively.
	\cite[Algorithm 4.1]{gidisu2022SIAMgeneralized} is a summary of this procedure. 
	
	\hskip 2em
	As summarized in \cite{gidisu2022SIAMgeneralized}, the overall complexity of the algorithm  is chiefly governed by the construction of the GSVD, which costs $\mathcal{O}((m+n+d)n^2)$.
	Nevertheless, this computational cost can be prohibitively expensive when the dimensions are very large, making it difficult for large-scale applications.
	To tackle the large-scale problems where a full GSVD may not be affordable, we turn to
	the randomized algorithms \cite{wei2021CAMCrandomized,halko2011SIAMfinding}, which are typically computationally efficient and easy to implement.
	Moreover, they have favorable numerical properties such as stability, and allow for restructuring computations in ways that make them amenable to implementation in a variety of settings including parallel computations.
	Following this success, and building on the random sampling techniques \cite{halko2011SIAMfinding}, we develop randomized algorithms for efficiently computing the GCUR, and a more exhaustive treatment for our randomized approaches-including pseudocode, and the detailed error analysis will be discussed in the following work.

	\subsection{Randomization for DEIM Based GCUR}\label{subsection: Randomization for the DEIM based GCUR}
	\hskip 2em
	As concluded in \cite{halko2011SIAMfinding}, there are two main steps involved in the calculation of a low-rank approximation to a given matrix $A$.
	The first stage involves constructing a low-dimensional subspace that captures the principal action of the input matrix, which can be executed very efficiently with random sampling methods.
	In other words, we need a matrix $Q$ for which
	\begin{equation*}
		Q\  \mathrm{has}\ \mathrm{orthonormal}\ \mathrm{columns}\ \mathrm{and}\
		A\approx QQ^{\mathrm{T}}A.
	\end{equation*}
	The second is to restrict the matrix to the subspace and subsequently perform a	standard factorization (QR, SVD, etc.) on the reduced matrix,
	and it can be  reliably executed using established deterministic techniques.
	Here we wish to compute the approximate GSVD of the input pair $(A,B)$, where $A\in \mathbb{R}^{m\times n}$, $B\in \mathbb{R}^{d\times n}$ with $m \geq n$, such that
	\begin{equation}\label{RGSVD}
		\left[\begin{array}{l}
			B \\
			A
		\end{array}\right] \approx\left[\begin{array}{c}
			B  \\
			Q Q^{\mathrm{T}} A
		\end{array}\right]=\left[\begin{array}{ll}
			V & \\
			& U
		\end{array}\right]\left[\begin{array}{l}
			\Sigma \\
			\Gamma
		\end{array}\right]
			Y^{\mathrm{T}}.
	\end{equation}
	This goal can be achieved after	five simple steps \cite{wei2016SIAMtikhonov}:
	
	\hskip 2em 1. Generate an $n \times (k+p)$ Gaussian random matrix $\Omega$;
	
	\hskip 2em 2. Form the $m\times (k+p)$ matrix $K=A\Omega$;
	
	\hskip 2em 3. Compute the $m\times (k+p)$ orthonormal matrix $Q$ via the QR factorization
				  $K = QR$;
	
	\hskip 2em 4. Compute the GSVD of $(Q^{\mathrm{T}}A,B)$:
	$\left[\begin{array}{c} B \\ Q^{\mathrm{T}} A \end{array}\right]
 	 =\left[\begin{array}{ll} V & \\ & W \end{array}\right]
 	  \left[\begin{array}{l} \Sigma \\ \Gamma \end{array}\right] Y^{\mathrm{T}}$;
	
	\hskip 2em 5. Form the $m\times (r+p)$ matrix $U = QW$.
	
	By \cite{halko2011SIAMfinding}, the above operations generates (\ref{RGSVD}) with the error $E=A-QQ^{\mathrm{T}}A$ saitisfying
	\begin{equation}\label{Err of RGSVD}
		\left\|E\right\|
		\leq
		\left(1+6 \sqrt{(k+p) p \log p}\right) {\sigma}_{k+1}(A)
		+3 \sqrt{k+p} \sqrt{\sum_{j>k} {\sigma}_j^2(A)}
	\end{equation}
	with probability not less than $1-3p^{-p}$, where $\sigma_{j}(A)$ is the $j$th largest singular value of $A$.
	Here $p$ is the oversampling parameter, which usually determines that a small number of columns are added to provide flexibility \cite{halko2011SIAMfinding}, and the proper selection of $p$ is imperative for optimal algorithm performance.
 	The primary computational expense of the randomized approach stems from the computation of the GSVD for the significantly smaller matrix pair $(Q^{\mathrm{T}}A,B)$.
 	
 	\hskip 2em
 	Combining the randomized GSVD algorithm with the DEIM technique, we present our randomized algorithm for computing the GCUR decomposition in Algorithm \ref{Al-R-DEIM-GCUR}.
 		\begin{algorithm}[htb]
 		\caption{DEIM based GCUR randomized algorithm}
 		\label{Al-R-DEIM-GCUR}
 		\hspace*{0.02in} {\bf Input:}
 		$A \in \mathbb{R}^{m \times n}$
 		and
 		$B \in \mathbb{R}^{m \times n}$ with $m\ge n$ and $d\ge n$, desired rank $k$, and the oversampling
 		\hspace*{0.02in} parameter $p$.
 		\\
 		\hspace*{0.02in} {\bf Output:}
 		A rank-$k$ GCUR decomposition\\
 		\hspace*{0.02in}
 		$\hat{A}=A(:,\mathbf{p}) \cdot M_A \cdot A(\mathbf{s}_A,:)$,
 		$\hat{B}=B(:,\mathbf{p}) \cdot M_B \cdot B(\mathbf{s}_B,:)$.
 		\begin{algorithmic}[1]
 			\State Generate an $n\times (k+p)$ Gaussian random matrix $\Omega$.
 			\State Form the $m\times (k+p)$ matrix $K=A\Omega$.
 			\State Compute the $m\times (k+p)$ orthonormal matrix $Q$ via the QR factorization
 			$K = QR$.
 			\State Compute the GSVD of $(B,Q^{\mathrm{T}}A)$:
 			$\left[\begin{array}{c} B \\  Q^{\mathrm{T}} A \end{array}\right]
 			=\left[\begin{array}{ll} V & \\ & W \end{array}\right]
 			\left[\begin{array}{l} \Sigma \\ \Gamma \end{array}\right] Y^{\mathrm{T}}$.
 			\State Form the $m\times (r+p)$ matrix $U = QW$.
 			\State  $y=Y(:, 1)$.
 			\State $p_{1}=\operatorname{argmax}_{1 \leq i \leq n}\left|y_{i}\right|$.
 			\State $\mathbf{p}=[\begin{array}{l}
 				p_1
 			\end{array}]$.
 			\State $\mathbf{p}=\mathrm{deim}(Y)$.
 			\State $\mathbf{s}_A=\mathrm{deim}(U)$.
 			\State $\mathbf{s}_B=\mathrm{deim}(V)$.
 			\State
			Compute
 			$M_A=A(:, \mathbf{p}) \backslash\left(A / A\left(\mathbf{s}_A,:\right)\right)$,
 			$M_B=B(:, \mathbf{p}) \backslash\left(B / B\left(\mathbf{s}_B,:\right)\right)$.
 		\end{algorithmic}
 	\end{algorithm}
 	The proposed Algorithm \ref{Al-R-DEIM-GCUR} takes advantage of randomization techniques \cite{wei2016SIAMtikhonov} to accelerate the GSVD process
 	and obtain the generalized singular vectors efficiently.
 	Then we use the DEIM index selection procedure, operating on the approximate generalized singular vector matrices to determine the selected columns and rows.
 	We note that we can parallelize the	work in lines 7 to 15 since it consists of three independent runs of DEIM.
 	Additionally, if the objective is to approximate the matrix $A$ from the pair $(A, B)$, the manipulation on $V$ can be omitted as noted in \cite{gidisu2022SIAMgeneralized}.
	 It is worth mentioning that the dominant cost of the randomized algorithm lies in computing the GSVD of the matrix pair $(Q^{\mathrm{T}}A,B)$, 
 	which is significantly lower than its counterpart in the non-random algorithm.
 	Consequently, from a practical perspective, Algorithm \ref{Al-R-DEIM-GCUR} is extremely simple to implement and can greatly reduce the computational time.
 	The following work will give performance guarantees by quantifying the error of the rank-$k$ GCUR decomposition
 	$\hat{A}=C_A M_A R_A = A(:, \mathbf{p}) \cdot M_A \cdot A\left(\mathbf{s}_A,:\right)$
 	and
 	$\hat{B}=C_B M_B R_B = B(:, \mathbf{p}) \cdot M_B \cdot B\left(\mathbf{s}_B,:\right)$.
 	
 	\hskip 2em
 	Consistent with (\ref{GSVD of A}) and (\ref{GSVD of B}), let the ordered number pairs $\{(\gamma_i, \beta_i)\}_{i=1}^n$ be the generalized singular values of the matrix pair $(A,B)$, where we arrange the ratios $\gamma_i / \beta_i$ in a non-increasing order.
	 As outlined in Algorithm \ref{Al-R-DEIM-GCUR}, the matrix pair $(A,B)$ owns an approximate GSVD
	\begin{equation*}
		QQ^{\mathrm{T}}A=U\Gamma Y^{\mathrm{T}} \qquad
		\mathrm{and} \qquad
		B=V\Sigma Y^{\mathrm{T}},
	\end{equation*}
	where
	$\Gamma=\operatorname{diag}(\tilde{\gamma}_1, \ldots, \tilde{\gamma}_n)$,
	$\Sigma=\operatorname{diag}(\tilde{\beta}_1, \ldots, \tilde{\beta}_n)$,
	and the ratios $\tilde{\gamma}_i / \tilde{\beta}_i$ are in non-increasing order,
	and the approximation error satisfies (\ref{Err of RGSVD}) with failure probability not exceeding $3p^{-p}$.
	Partition the matrices:
	\begin{equation}\label{partition of GCUR eq1}
		U=\left[ \begin{array}{ll}	U_k &\widehat{U}  \end{array}\right],~~~
		V=\left[ \begin{array}{ll}	V_k & \widehat{V} \end{array}\right],~~~
		Y=\left[\begin{array}{ll}	Y_k & \widehat{Y} \end{array}\right],
	\end{equation}
	\begin{equation}\label{partition of GCUR eq2}
		\Gamma=\operatorname{diag}\left(\Gamma_k, ~\widehat{\Gamma}\right),~~~
		\Sigma=\operatorname{diag}\left(\Sigma_k, ~\widehat{\Sigma}\right),
	\end{equation}
	where matrices $U_k$, $V_k$, and $Y_k$ contain the first $k$ columns of $U$, $V$, and $Y$ respectively.
	For our analysis, instead of $Y$, we use its orthonormal QR factor $H$ 
    from the QR decomposition of $Y$:
	\begin{equation}\label{partition of GCUR eq3}
		\left[\begin{array}{ll}
			Y_k & \widehat{Y}
		\end{array}\right]=Y=H T=\left[\begin{array}{ll}
			H_k & \widehat{H}
		\end{array}\right]\left[\begin{array}{cc}
			T_k & T_{12} \\
			0 & T_{22}
		\end{array}\right]=\left[\begin{array}{ll}
			H_k T_k & H \widehat{T}
		\end{array}\right],
	\end{equation}
with $\widehat{T} =\left[\begin{array}{l}
			T_{12} \\
			T_{22}
		\end{array}\right].$
	This implies that
	$
		QQ^{\mathrm{T}}A
		=U_k \Gamma_k Y_k^{\mathrm{T}}+\widehat{U} \widehat{\Gamma} \widehat{Y}^{\mathrm{T}}
		=U_k \Gamma_k T_k^{\mathrm{T}} H_k^{\mathrm{T}}+\widehat{U} \widehat{\Gamma} \widehat{T}^{\mathrm{T}} H^{\mathrm{T}}.
	$ 
	With the above preparation, the following theorem derives the error bound for $\|A-\hat{A}\|$.
	\begin{theorem}\label{Err of R-DEIM-GCUR}
	Suppose $A\in\mathbb{R}^{m\times n},~B\in\mathbb{R}^{d\times n}$ and both are of full column rank, and let matrix pair $(\hat{A}, \hat{B})$ be a rank-$k$ GCUR decomposition for matrix pair $(A,B)$ computed by Algorithm \ref{Al-R-DEIM-GCUR}. Let
	$\Theta_k = (1+6 \sqrt{(k+p) p \log p}) \sigma_{k+1}(A)+3 \sqrt{k+p} \sqrt{\sum_{j>k} \sigma_j^2(A)}$, and
	$\eta_k = \sqrt{\frac{nk}{3}}2^k + \sqrt{\frac{mk}{3}}2^k$.
	Then
	\begin{equation}\label{Err of R-DEIM-GCUR for A}
		\|\hat{A}-A \|
		\le
		\eta_k
     \left[
		\Theta_k
		+
		\left(\left\|A\right\|+\left\|B\right\|\right)
		\left(
			\frac{\gamma_{k+1}}{\beta_{k+1}} + \frac{\Theta_k}{\beta_{k+1}}
			\left\|
			\left(\begin{array}{l}
				A \\
				B
			\end{array}\right)^{\dagger}			
			\right\|
		\right)
		\right]
	\end{equation}
	holds with probability not less than $1-3p^{-p}$,
	where the number pair $(\gamma_{k+1},\beta_{k+1})$
    is defined in (\ref{GSVD of B}) and (\ref{GSVD of A}), and both ${\gamma}_i / {\beta}_i$ and $1/\beta_{i}$ are in a non-increasing order.
	\end{theorem}
	\begin{proof}
	 	By the definition of $M_A$, we have
	 	\begin{equation*}
	 		A-C_A M_A R_A=A-C_AC_A^{\dagger}AR_A^{\dagger}R_A=
	 		(I-C_A C_A^{\dagger}) A+C_A C_A^{\dagger} A(I-R_A^{\dagger} R_A).
	 	\end{equation*}
 	Since $C_AC_A^{\dagger}$ is an orthogonal projection, it directly follows that
 	\begin{equation}\label{Proof R-DEIM-GCUR eq1}
 		\begin{aligned}
 			\| A-C_AM_AR_A \|
 			\le&
 			\| (I-C_A C_A^{\dagger}) A \|+ \| A(I-R_A^{\dagger} R_A)\|.
 		\end{aligned}
 	\end{equation}
    According to \cite[Lemma 3.2]{sorensen2016SIAMdeim}, the column and row indices $\mathbf{s}_A$ and $\mathbf{p}$ give the full rank matrices $C_A=AS_A$ and $R_A=P^{\mathrm{T}}A$ where
		 $S_A = I(:, \mathbf{s}_A)$ and $P=I(:,\mathbf{p})$.
	Let
	$\mathbb{P}=P(H_k^{\mathrm{T}}P)^{-1}H_k^{\mathrm{T}}$ and
	$\mathbb{S}=U_k(S_A^{\mathrm{T}}U_k)^{-1}S_A^{\mathrm{T}}.$
	Then using the result in \cite[Proposition 4.7]{gidisu2022SIAMgeneralized}, we get
	\begin{equation*}
		\|(I-C_AC_A^{\dagger})A\|
		\le
		\|A(I-\mathbb{P})\|,~~~
		\|A(I-R_A^{\dagger} R_A)\|
		\le
		\|(I-\mathbb{S})A\|.
	\end{equation*}
	Note that $U_k^{\mathrm{T}}U_k=I$ and $H_k^{\mathrm{T}}H_k=I$.
	Then according to \cite[Lemma 4.1]{sorensen2016SIAMdeim}, we obtain that
	\begin{equation}\label{Proof R-DEIM-GCUR eq2}
		\begin{aligned}
			\|(I-C_AC_A^{\dagger})A\|
			\le&
			\|(H_k^{\mathrm{T}}P)^{-1}\|    \|A(I-H_k H_k^{\mathrm{T}})\|
			\\ \le &
			\|(H_k^{\mathrm{T}}P)^{-1}\|    \left( \|E\| +\|QQ^{\mathrm{T}}A\left(I-H_kH_k^{\mathrm{T}}\right)\|   \right),
		\end{aligned}
	\end{equation}
    where we use $\left\|I-H_kH_k^{\mathrm{T}}\right\|=1.$
 	Analogous operation gives that
 	\begin{equation}\label{Proof R-DEIM-GCUR eq3}
 		\|A(I-R_A^{\dagger}R_A)\|
 		\le
 		\|(S_A^{\mathrm{T}}U_k)^{-1}\|
 		\left( \|E\| + \|(I-U_kU_k^{\mathrm{T}})QQ^{\mathrm{T}}A\| \right).
 	\end{equation}
 	Note that
 	\begin{equation*}
 		\begin{aligned}
 			QQ^{\mathrm{T}}A H_k H_k^{\mathrm{T}}
 			&=
 			\left[\begin{array}{ll}
 				U_k & \widehat{U}
 			\end{array}\right]
 			\left[\begin{array}{cc}
 				\Gamma_k & 0 \\
 				0 & \widehat{\Gamma}
 			\end{array}\right]
 			\left[\begin{array}{cc}
 				T_k^{\mathrm{T}} & 0 \\
 				T_{12}^{\mathrm{T}} & T_{22}^{\mathrm{T}}
 			\end{array}\right]
 			\left[\begin{array}{c}
 				I_k \\
 				0
 			\end{array}\right] H_k^{\mathrm{T}}
 			\\&=
 			U_k \Gamma_k T_k^{\mathrm{T}} H_k^{\mathrm{T}}
 			+\widehat{U} \widehat{\Gamma} T_{12}^{\mathrm{T}} H_k^{\mathrm{T}},
 		\end{aligned}
 	\end{equation*}
  	and hence,
  	\begin{equation*}
  			QQ^{\mathrm{T}}A\left(I-H_k H_k^{\mathrm{T}}\right)
  			=\widehat{U} \widehat{\Gamma} \widehat{T}^{\mathrm{T}} H^{\mathrm{T}}
  			  -\widehat{U} \widehat{\Gamma} T_{12}^{\mathrm{T}} H_k^{\mathrm{T}}
  			 =\widehat{U} \widehat{\Gamma} T_{22}^{\mathrm{T}} \widehat{H}^{\mathrm{T}}.
  	\end{equation*}
    Similarly, it holds that
  	\begin{equation*}
  	\left(I-U_k U_k^T\right)QQ^{\mathrm{T}}A
  	=QQ^{\mathrm{T}}A-U_k \Gamma_k Y_k^{\mathrm{T}}
  	=\widehat{U} \widehat{\Gamma} \widehat{Y}^{\mathrm{T}}
  	=\widehat{U} \widehat{\Gamma} \widehat{T}^T H^{\mathrm{T}}.
 	 \end{equation*}
  	Therefore,
  	\begin{equation}\label{Proof R-DEIM-GCUR eq4}
  			\|QQ^{\mathrm{T}}A(I-HH_k^{\mathrm{T}})\|
  			=
  			\|\widehat{U} \widehat{\Gamma} T_{22}^{\mathrm{T}} \widehat{H}^{\mathrm{T}}\|
  			\le
  			\tilde{\gamma}_{k+1}\|T_{22}\|
  			\le
  			\tilde{\gamma}_{k+1}\|\widehat{T}\|,
  	\end{equation}
  	\begin{equation}\label{Proof R-DEIM-GCUR eq5}
  		\|(I-U_kU_k^{\mathrm{T}})QQ^{\mathrm{T}}A\|
  		\le
  		\tilde{\gamma}_{k+1}\|\widehat{T}\|.
  	\end{equation}
  	To bound $\|\widehat{T}\|$, recall the result in \cite[Theorem 2.3]{hansen1998SIAMrank} that
  	$\|Y\|\le\|QQ^{\mathrm{T}}A\|+\|B\|$.
  	Given the partitioning and QR factorization of $Y$, we have
  	\begin{equation}\label{Proof R-DEIM-GCUR eq6}
  		\|\widehat{T}\|=\|H \widehat{T}\|=\|\widehat{Y}\| \leq\|Y\| \leq\|QQ^{\mathrm{T}}A\|+\|B\|
  		\le \|A\|+\|B\|.
  	\end{equation}
	For the DEIM selection scheme, \cite[Lemma 4.4]{sorensen2016SIAMdeim} derives the bound
	\begin{equation}\label{Proof R-DEIM-GCUR eq8}
		\left\|\left(H_k^{\mathrm{T}} P\right)^{-1}\right\|<\sqrt{\frac{n k}{3}} 2^k,
	    \quad
	    \text { and }
	    \left\|\left(S_A^{\mathrm{T}} U_k\right)^{-1}\right\|<\sqrt{\frac{m k}{3}} 2^k.
	\end{equation}
	Inserting (\ref{Proof R-DEIM-GCUR eq2})-(\ref{Proof R-DEIM-GCUR eq8}) and  into (\ref{Proof R-DEIM-GCUR eq1}),
    we obtain
		\begin{equation}\label{Proof R-DEIM-GCUR eq9}
			\|\hat{A}-A \|
			\le
			\eta_k \left(
             \|E\|   +
			\tilde{\gamma}_{k+1}
			\left(\left\|A\right\|+\left\|B\right\|\right)
			\right),
	\end{equation}
    where $\tilde{\gamma}_{k+1}$ is the $(k+1)$th diagonal entry of $\Gamma$ with a non-increasing order.

 	\hskip 2em
 	Recall the perturbation results for the generalized singular values in \cite[Theorem 3]{sun1982MNSperturbation}
	\begin{equation*}
		\left|\tilde{\gamma}_i\beta_i-\tilde{\beta}_i\gamma_i\right|
		\le
		\left\|
		\left(\begin{array}{l}
			E \\
			F
		\end{array}\right)			
		\right\|
		\cdot
		\left\|
		\left(\begin{array}{l}
			A \\
			B
		\end{array}\right)^{\dagger}			
		\right\|, \quad 1 \leqslant i \leqslant n,
	\end{equation*}
	where matrices $E$ and $F$ are the perturbations to $A$ and $B$, respectively.
	Clearly, we have $F=0$ for our randomized algorithm.
	As a result, we have
	\begin{equation}\label{Err of tilde alpha_k+1}
		\tilde{\gamma}_{k+1}
		\le
		\frac{1}{\beta_{k+1}}
		\left(
			\gamma_{k+1}
			+
			\|E\|
			\cdot
			\left\|
			\left(\begin{array}{l}
				A \\
				B
			\end{array}\right)^{\dagger}			
			\right\|
		\right).
	\end{equation}
	We finish the proof by combining (\ref{Proof R-DEIM-GCUR eq9}), (\ref{Err of tilde alpha_k+1}) and the probabilistic error bound (\ref{Err of RGSVD}).
	\end{proof}
	Because the ratios ${\gamma}_i / {\beta}_i$ and $1/\beta_{i}$ are maintained in a non-increasing order, the right-hand side of (\ref{Err of R-DEIM-GCUR for A}) decreases as the target rank $k$ increases.
 Note that the randomized GSVD algorithm provides an exact decomposition of $B$, the error bound for $\|B-\hat{B }\|$ in \cite{gidisu2022SIAMgeneralized} still holds that
	\begin{equation*}
		\begin{aligned}
			\left\|B-C_B M_B R_B\right\|
			& \leq
			\|\left(H_k^{\mathrm{T}} P\right)^{-1}\|
			\cdot
			\left\|T_{22}\right\|
			+\|\left(S_B^{\mathrm{T}} V_k\right)^{-1}\|
			\cdot
			\|\widehat{T}\| \\
			& \leq
			\left(\|\left(H_k^{\mathrm{T}} P\right)^{-1}\|
			+\|\left(S_B^{\mathrm{T}} V_k\right)^{-1}\|\right)
			\cdot\|\widehat{T}\| \\
			&\leq
			\left(\sqrt{\frac{n k}{3}} 2^k + \sqrt{\frac{d k}{3}} 2^k \right)
			\left(\|A\|+\|B\|\right).
		\end{aligned}
	\end{equation*}
	Compared with the error bound of $\|A-C_A M_A R_A\|$ under the non-random scheme in \cite{gidisu2022SIAMgeneralized} that
	\begin{equation*}
		\|A-C_A M_A R_A\| \le {\gamma}_{k+1}(\|A\|+\|B\|) \cdot \eta_k,
	\end{equation*}
	(\ref{Err of R-DEIM-GCUR for A}) involves a truncation term
	$\Theta_k$  due to the randomization of the GSVD, and consequently, our randomized approach works well for matrices whose singular
	values exhibit some decay.

	\subsection{Randomization for L-DEIM Based GCUR}
	\hskip 2em
	To enhance the efficiency of our randomized algorithm, we now turn our gaze to combining the random sampling methods with the L-DEIM algorithm.
	The resulting technique, described in Algorithm \ref{Al-R-LDEIM-GCUR}, delivers acceptable error bounds with a high degree of probability, while also reducing computational costs.
	The associated probabilistic error estimate is presented in Theorem \ref{Err of R-LDEIM-GCUR}.
%

	\begin{algorithm}[htb]
		\caption{L-DEIM based GCUR randomized algorithm}
		\label{Al-R-LDEIM-GCUR}
		\hspace*{0.02in} {\bf Input:}
		$A \in \mathbb{R}^{m \times n}$
		and
		$B \in \mathbb{R}^{m \times n}$ with $m\ge n$ and $d\ge n$, desired rank $k$, the oversampling
		\hspace*{0.02in} parameter $p$ and the specified parameter $\widehat{k}$.
		\\
		\hspace*{0.02in} {\bf Output:}
		A rank-$k$ GCUR decomposition\\
		\hspace*{0.02in}
		$\hat{A}=A(:,\mathbf{p}) \cdot M_A \cdot A(\mathbf{s}_A,:)$,
		$\hat{B}=B(:,\mathbf{p}) \cdot M_B \cdot B(\mathbf{s}_B,:)$.
		\begin{algorithmic}[1]
			\State Generate an $n\times (\widehat{k}+p)$ Gaussian random matrix $\Omega$.
			\State Form the $m\times (\widehat{k}+p)$ matrix $K=A\Omega$.
			\State Compute the $m\times (\widehat{k}+p)$ orthonormal matrix $Q$ via the QR factorization
			$K = QR$.
			\State Compute the GSVD of $(Q^{\mathrm{T}}A,B)$:
			$\left[\begin{array}{c} B \\ Q^{\mathrm{T}} A\end{array}\right]
			=\left[\begin{array}{ll} V & \\ & W \end{array}\right]
			\left[\begin{array}{l}\Sigma \\ \Gamma \end{array}\right] Y^{\mathrm{T}}$.
			\State Form the $m\times (\widehat{k}+p)$ matrix $U = QW$.
			\State $\mathbf{p}=\mathrm{l\mbox{-}deim}(Y)$.
			\State $\mathbf{s}_A=\mathrm{l\mbox{-}deim}(U)$.
			\State $\mathbf{s}_B=\mathrm{l\mbox{-}deim}(V)$.
			%
			%
			\State
			Compute
			$M_A=A(:, \mathbf{p}) \backslash\left(A / A\left(\mathbf{s}_A,:\right)\right)$,
			$M_B=B(:, \mathbf{p}) \backslash\left(B / B\left(\mathbf{s}_B,:\right)\right)$.
		\end{algorithmic}
	\end{algorithm}
	
	\begin{theorem}\label{Err of R-LDEIM-GCUR}
		Let the matrix pair $(\hat{A},\hat{B})$ be a rank-$\widehat{k}$ GCUR approximation for pair $(A,B)$ computed by Algorithm \ref{Al-R-LDEIM-GCUR}.
		Suppose that
		$\Theta_{\widehat{k}} = \left(1+6 \sqrt{({\widehat{k}}+p) p \log p}\right) \sigma_{{\widehat{k}}+1}(A)+3 \sqrt{{\widehat{k}}+p} \sqrt{\sum_{j>{\widehat{k}}} \sigma_j^2(A)}$, and
		$\eta_{\widehat{k}} = \sqrt{\frac{n{\widehat{k}}}{3}}2^{\widehat{k}} + \sqrt{\frac{m{\widehat{k}}}{3}}2^{\widehat{k}}$,
		and then the following error bound
		\begin{equation*}
			\|\hat{A}-A \|
			\le
			\eta_{\widehat{k}}\left[
			\Theta_{\widehat{k}}
			+ \left(\left\|A\right\|+\left\|B\right\|\right)
			\left(
			\frac{\gamma_{\widehat{k}+1}}{\beta_{{\widehat{k}}+1}}  + \frac{\Theta_{\widehat{k}}}{\beta_{{\widehat{k}}+1}}
			\left\|
			\left(\begin{array}{l}
				A \\
				B
			\end{array}\right)^{\dagger}			
			\right\|
			\right)
			\right],
		\end{equation*}
		fails with probability not exceeding than $3p^{-p}$ and ${\gamma_i}/{\beta_i}$ and ${1}/{\beta_i}$ are in a non-increasing order.
	\end{theorem}
	\section{Randomization for RSVD-CUR}
	
	\hskip 2em
	In \cite{gidisu2022ARXIVrsvd}, Gidisu and Hochstenbach generalized the DEIM-type CUR to a new coordinated CUR decomposition 
	based on the RSVD, which was called the RSVD-CUR decomposition.
	This novel factorization presents a viable technique for reducing the dimensionality of multi-view datasets in the context of a two-view scenario, and can be also applied to the multi-label classification problems and specific types of perturbation recovery problems.
	In this section, we introduce new randomized algorithms for computing the RSVD-CUR decomposition where we apply the L-DEIM scheme and the random sampling techniques.
	Detailed error analysis which provides insight into the accuracy of the algorithms and the choice of the algorithmic parameters is given.
	\subsection{RSVD-CUR}
	\hskip 2em
	Given a matrix triplet $(A, B, G)$ with
	$A\in\mathbb{R}^{m\times n}$
	$B\in\mathbb{R}^{m\times \ell}$, and
	$G\in\mathbb{R}^{d\times n}$ ($\ell \ge d \ge m \ge n$), where $B$ and $G$ both have full rank.
	Then a rank-$k$ RSVD-GCUR approximation of $(A,B,G)$ provides the CUR-type low-rank approximations such that
	\begin{equation}\label{RSVD-CUR of (A,B,G)}
		\begin{aligned}
			A & \approx C_A M_A R_A=AP\  M_A\  S^{\mathrm{T}} A, \\
			B & \approx C_B M_B R_B=BP_B\  M_B\  S^{\mathrm{T}} B, \\
			G & \approx C_G M_G R_G=GP\  M_G\  S_G^{\mathrm{T}} G,
		\end{aligned}
	\end{equation}
	where the index selection matrices
	$S\in \mathbb{R}^{m \times k}$, $S_G \in \mathbb{R}^{d \times k},$
	$P\in \mathbb{R}^{n \times k}$, and $P_B\in \mathbb{R}^{\ell \times k}$
	are submatrices of the identity.
	
	\hskip 2em
	The matrices $C_A\in \mathbb{R}^{m \times k}$, $C_B\in \mathbb{R}^{m \times k}$, $C_G\in \mathbb{R}^{d \times k}$
	and $R_A\in \mathbb{R}^{k \times n}$, $R_B\in \mathbb{R}^{k \times \ell}$, $R_G\in \mathbb{R}^{k \times n}$ are formed from the rows or columns of the given matrices.
	Suppose that the sampling indices are stored in the vectors $\mathbf{s}$, $\mathbf{s}_G$, $\mathbf{p}$ and $\mathbf{p}_B$
	such that
	$S=I(:,\mathbf{s})$, $S_G=I(:,\mathbf{s}_G)$,
	$P=I(:,\mathbf{p})$, and $P_B=I(:,\mathbf{p}_B)$.
	In \cite{gidisu2022ARXIVrsvd}, the indices are selected according to the information contained in the orthogonal and nonsingular matrices from the rank-$k$ RSVD,
	where the DEIM and L-DEIM algorithms are employed as the index selection strategies for finding the optimal indices.
	Specifically, suppose that the RSVD of $(A,B,G)$ are available, as shown in (\ref{RSVD of (A,B,G)}).
	To construct a DEIM-type RSVD-CUR decomposition of a matrix pair $(A,B,G)$, given the target rank $k$, the DEIM operates on the first $k$ columns on matrices
	$W$, $Z$, $U$ and $V$ to obtain the corresponding indices
	$\mathbf{p}$, $\mathbf{s}$, $\mathbf{p}_B$ and $\mathbf{s}_G$.
	Moreover, by utilizing the L-DEIM, one can use at least the first $k/2$ vectors of $W$, $Z$, $U$, and $V$ to obtain the indices, with the approximation quality as good as that of the DEIM-type RSVD-CUR, which is demonstrated numerically in \cite{gidisu2022ARXIVrsvd}.
	
	\hskip 2em
	It is clear that both the DEIM and the L-DEIM type RSVD-CUR decompositions require the inputs of the RSVD.
	Nevertheless, computing this factorization can be a significant	computational bottleneck in the large-scale applications.
	How to reduce this computational cost and still ensure the accuracy of the approximation is our main concern.
	Next, we introduce the randomized schemes for computing the RSVD-CUR decomposition, together with detailed error analysis.
	\subsection{Randomization for Restricted SVD}
	\hskip 2em
	The RSVD \cite{chu2000SIAMcomputation,zhang2021NCneural,zwaan2020Arxivtowards} is a generalization of the ordinary singular value decomposition (OSVD) to matrix triplets, with the applications including rank minimization of structured perturbations, Gauss–Markov model and restricted total least squares, etc.
	The calculation of RSVD can be accomplished by two GSVDs.
	We first compute the GSVD of $(A,G)$,
	\begin{equation}\label{the first GSVD}
		A=U_1 \Gamma_1 Y_1^{\mathrm{T}},
		\quad
		G=V_1
			\left[\begin{array}{l}
			\Sigma_1 \\
			0_{d-n,n}
		\end{array}\right]
		Y_1^{\mathrm{T}},
	\end{equation}
	and then we compute the GSVD of $(B^\mathrm T U_1,\Sigma_1^{-1} \Gamma_1^\mathrm T)$, so that
	\begin{equation*}
		B^{\mathrm{T}} U_1 = U_2 \Gamma_2  Y_2^\mathrm{T},
		\quad
         \Sigma_1^{-1} \Gamma_1^\mathrm T = V_2 \Sigma_2 Y_2^\mathrm T.
	\end{equation*}	
	Critically, we keep the diagonal entries of $\Gamma_1$ and $\Gamma_2$ in
		non-decreasing order while those of $\Sigma_1$ and $\Sigma_2$ are non-increasing.
	In accordance with (\ref{RSVD of (A,B,G)}), one can define
	\begin{equation}\label{component of R-RSVD}
		Z \triangleq U_1Y_2,\
		W \triangleq Y_1\Sigma_1V_2\Gamma_G^{-1},\
		V \triangleq V_1 \widehat{V}_2,\
		U \triangleq U_2,\
	\end{equation}	
	\begin{equation*}
		\left\{
		\begin{aligned}
			D_A &\triangleq \Sigma_2^{\mathrm{T}} \Gamma_G
			=\left[\begin{array}{ccc}
				\alpha_1 & & \\
				& \ddots & \\
				&  & \alpha_n \\
				\hdashline & 0_{m-n,n} &
			\end{array}\right] \in \mathbb{R}^{ m\times n}, \quad
		D_G  \triangleq	\left[\begin{array}{c}
			\Gamma_G \\
			0_{d-n, n}
		\end{array}\right]
		=\left[\begin{array}{ccc}
			\gamma_1 & & \\
			& \ddots & \\
			&  & \gamma_n \\
			\hdashline & 0_{d-n,n} &
		\end{array}\right] \in \mathbb{R}^{d\times n},\\
			D_B &\triangleq \Gamma_2^{\mathrm{T}}
			=\left[\begin{array}{lll:l:l}
				\beta_1 & & & & \\
				& \ddots& & 0_{n, m-n} & 0_{n, l-m} \\
				& & \beta_n & & \\
                \hdashline
                & 0_{m-n, n} &   & I_{m-n} & 0_{m-n, l-m}
		\end{array}\right] \in \mathbb{R}^{ m\times l}, 
		\end{aligned}
		\right.	
	\end{equation*}
	where $\Gamma_G = \operatorname{diag}\left(\gamma_1, \ldots, \gamma_n\right) \in \mathbb{R}^{n \times n}$ is a scaling matrix one can freely choose, as shown in \cite{zwaan2020Arxivtowards}.
	Denote $\Sigma_2 = \operatorname{diag}\left(\sigma_1, \ldots, \sigma_n\right) \in \mathbb{R}^{n \times m}$.
	Here we set
	$\gamma_i=\frac{\sigma_i}{\sqrt{\sigma_i^2+1}}$ for $i=1,\ldots, n,$ which are ordered non-increasingly
	then it follows that
	$\alpha_i=\frac{\sigma_i^2}{\sqrt{\sigma_i^2+1}}$.
	From the second GSVD, we have $\beta_i^2+\sigma_i^2=1$, then it leads that
	$\alpha_i^2+\beta_i^2+\gamma_i^2=1$ for $i=1, \ldots, n$.
	Note that the matrices $B$ and $G$ are of full rank, then we have
	$1>\alpha_i\ge\alpha_{i+1}>0$, $1>\gamma_i\ge\gamma_{i+1}>0$ and $0<\beta_i\le\beta_{i+1}<1$.

	\hskip 2em
	We now proceed to propose a fast randomized algorithm for computing the RSVD.
	The main idea of our approach is to accelerate this computational process by exploiting the randomized GSVD algorithm and its analysis relies heavily on the results introduced in Subsection \ref{subsection: Randomization for the DEIM based GCUR}.
	Firstly, an orthonormal matrix $H_1\in\mathbb{R}^{d\times (k+p_1)}$ is generated to satisfy	$\left\|G-H_1 H_1^{\mathrm{T}} G\right\| \leq c {\sigma}_{k+1}$
	with high probability, where $\sigma_{k+1}$ is the
	$(k+1)$th largest singular value of $G$ and $c$ is a constant depending on $k$ and $p_1$.
	Here $p_1$ is the oversampling parameter used to provide flexibility \cite{halko2011SIAMfinding}.
	According to (\ref{the first GSVD}),
	$\Sigma_1$ is required to be square, hence, here we fix that $p_1=n-k$.
	By performing the GSVD of
	$[(H_1^{\mathrm{T}} G)^{\mathrm{T}}, A^{\mathrm{T}}]^{\mathrm{T}}$,
	we get the approximate GSVD of $[G^{\mathrm{T}}, A^{\mathrm{T}}]^{\mathrm{T}}$,
	\begin{equation}\label{R-RSVD-GSVD Stage1}
		\left[\begin{array}{l}
			A \\
			G
		\end{array}\right] \approx\left[\begin{array}{c}
			A \\
			H_1 H_1^{\mathrm{T}} G
		\end{array}\right]=\left[\begin{array}{ll}
			U_1 & \\
			& V_1
		\end{array}\right]\left[\begin{array}{l}
			\Gamma_1 \\
			\Sigma_1
		\end{array}\right] Y_1^{\mathrm{T}}.
	\end{equation}
	When $m\gg n$, the computational advantage of (\ref{R-RSVD-GSVD Stage1}) becomes much more obvious.
	Furthermore, we can formulate the approximate GSVD for the pair
	$(B^{\mathrm{T}}U_1,\Sigma_1^{-1}\Gamma_1^\mathrm T)$
	by performing the GSVD of the small-scale matrix
	$[(H_2^{\mathrm{T}}B^{\mathrm{T}} U_1)^{\mathrm{T}}, (\Sigma_1^{-1} \Gamma_1^\mathrm T)^{\mathrm{T}}]^{\mathrm{T}}$,
	where $H_2$ is a ${(k+p_2)\times n}$ orthonormal matrix, and $p_2$ is also an oversampling parameter.
	Then we obtain
	\begin{equation}\label{R-RSVD-GSVD Stage2}
		\left[\begin{array}{c}
			B^{\mathrm{T}}U_1 \\
			\Sigma_1^{-1} \Gamma_1^\mathrm T
		\end{array}\right] \approx\left[\begin{array}{c}
			H_2H_2^{\mathrm{T}} (B^{\mathrm{T}}U_1) \\
			\Sigma_1^{-1} \Gamma_1^\mathrm T
		\end{array}\right]=\left[\begin{array}{ll}
			U_2 & \\
			& V_2
		\end{array}\right]\left[\begin{array}{c}
			\Gamma_2 \\
			\Sigma_2
		\end{array}\right] Y_2^{\mathrm{T}} .
	\end{equation}
	Finally, we can formulate the corresponding approximate RSVD of $(A,B,G)$,
	\begin{equation}\label{R-RSVD}
		A=Z D_A W^{\mathrm{T}},
		\qquad
		B\approx \tilde{B} = Z D_B U^{\mathrm{T}},
		\qquad
		G\approx \tilde{G} = V D_G W^{\mathrm{T}}.
	\end{equation}
	To be more clear in presentation, the above process can be expressed as follows:
		\begin{align*}
			{\left[\begin{array}{ll}
					A & B \\
					G &
				\end{array}\right] }
			&\approx
			{\left[\begin{array}{ll}
					A & B  \\
					H_1H_1^{\mathrm{T}}G &
				\end{array}\right] }
            =\left[\begin{array}{ll}
				U_1 & \\
				& V_1
			\end{array}\right]\left[\begin{array}{ll}
				\Gamma_1 & U_1^{\mathrm{T}} B \\
				\Sigma_1 &
			\end{array}\right]\left[\begin{array}{cc}
				Y_1^{\mathrm{T}} & \\
				& I
			\end{array}\right]
			\\
			&=\left[\begin{array}{ll}
				U_1 & \\
				& V_1
			\end{array}\right]\left[\begin{array}{cc}
				\Gamma_1 \Sigma_1^{-1} & U_1^{\mathrm{T}} B \\
				I &
			\end{array}\right]\left[\begin{array}{cc}
				\Sigma_1 Y_1^{\mathrm{T}} & \\
				& I
			\end{array}\right]
            \\
			&\approx
			\left[\begin{array}{ll}
				U_1 & \\
				& V_1
			\end{array}\right]\left[\begin{array}{cc}
				\Gamma_1 \Sigma_1^{-1} & \left(U_1^{\mathrm{T}} B\right)H_2H_2^\mathrm T\ \\
				I &
			\end{array}\right]\left[\begin{array}{cc}
				\Sigma_1 Y_1^{\mathrm{T}} & \\
				& I
			\end{array}\right]
			\\
			&=\left[\begin{array}{ll}
				U_1 Y_2 & \\
				& V_1
			\end{array}\right]\left[\begin{array}{cc}
				\Sigma_2^{\mathrm{T}} & \Gamma_2^{\mathrm{T}} \\
				V_2 &
			\end{array}\right]\left[\begin{array}{ll}
				V_2^{\mathrm{T}} \Sigma_1 Y_1^{\mathrm{T}} & \\
				& U_2^{\mathrm{T}}
			\end{array}\right]\\
			&=\left[\begin{array}{ll}
				U_1 Y_2 & \\
				& V_1 V_2
			\end{array}\right]\left[\begin{array}{cc}
				\Sigma_2^{\mathrm{T}} \Gamma_G & \Gamma_2^{\mathrm{T}} \\
				\Gamma_G &
			\end{array}\right]\left[\begin{array}{cc}
				Y_1 \Sigma_1 V_2 \Gamma_G^{-1} & \\
				& U_2
			\end{array}\right]^{\mathrm{T}}\\
			&\triangleq \left[\begin{array}{ll}
				Z & \\
				& V
			\end{array}\right]\left[\begin{array}{ll}
				D_A & D_B \\
				D_G &
			\end{array}\right]\left[\begin{array}{ll}
				W & \\
				& U
			\end{array}\right]^{\mathrm{T}} .
		\end{align*}
	We summarize the details in Algorithm \ref{Al-R-RSVD}.
	Notice that (\ref{R-RSVD}) indicates that our randomized approach provides an exact factorization for $A$, as a direct consequence of (\ref{R-RSVD-GSVD Stage1}),
	while it does not hold for matrices $B$ and $G$.
	We present a detailed analysis of the approximation error in the following theorem.
	\begin{theorem}
		Suppose that
		$B \in \mathbb{R}^{m \times l}$ and $G \in \mathbb{R}^{d \times n}$ with $l\ge d\ge m \ge n$
		and $p$ is an oversampling parameter.
		Let  $\tilde{B}$ and $\tilde{G}$ be the approximation of $B$ and $G$ computed by Algorithm \ref{Al-R-RSVD},
		then
	\begin{equation}\label{Err of R-RSVD-B}
		\|B-\tilde{B}\|
		\le
		\left(1+6 \sqrt{(k+p) p \log p}\right) \sigma_{k+1}(B)+ 3 \sqrt{k+p} \sqrt{\sum_{j>k} \sigma_j^2(B)},
	\end{equation}
	\begin{equation}\label{Err of R-RSVD-G}
		\|G-\tilde{G}\|
		\le
		\left(1+6 \sqrt{n (n-k) \log (n-k)}\right) \sigma_{k+1}(G) +3 \sqrt{n\sum_{j>k} \sigma_j^2(G)}
	\end{equation}
	hold with probability not less than $1-3p^{-p}$ and $1-(n-k)^{-(n-k)}$ respectively.
	\end{theorem}
	\begin{proof}
      Let $E_G$ and $E_B$ be the error matrices such that
		\begin{equation}\label{Proof of R-RSVD eq1}
			G=V_1\Sigma_1Y_1^{\mathrm{T}}+E_G,
			\qquad
			B^{\mathrm{T}}U_1=U_2\Gamma_2Y_2^{\mathrm{T}}+E_B,
			\qquad
			\Sigma_1^{-1} \Gamma_1^\mathrm T = V_2\Sigma_2Y_2^{\mathrm{T}}.
		\end{equation}
	Inserting (\ref{component of R-RSVD}) and (\ref{Proof of R-RSVD eq1}) into $\tilde{B}$ and $\tilde{G}$, we have
		\begin{equation*}
				B-\tilde{B}
				=
				B-ZD_BU^{\mathrm{T}}
				=
				B-(U_1Y_2)\Gamma_2^{\mathrm{T}}U_2^{\mathrm{T}}
				=
				B-U_1(U_1^{\mathrm{T}}B-E_B^{\mathrm{T}})
				=
				U_1E_B^{\mathrm{T}},
		\end{equation*}
		\begin{equation*}
				G-\tilde{G}
				=
				G-VD_GW^{\mathrm{T}}
				=
				G-(V_1V_2)\Gamma_G(Y_1\Sigma_1V_2\Gamma_G^{-1})^{\mathrm{T}}
				=
				G-V_1\Sigma_1Y_1^{\mathrm{T}}
				=
				E_G.
	\end{equation*}
	During randomization for the GSVD of $(A, G)$, we set the oversampling parameter $p'=n-k$.
	By the probabilistic error bound (\ref{Err of RGSVD}), we have
	\begin{equation*}
		\|G-\tilde{G}\| \le \|E_G\|
		\le
		\left(1+6 \sqrt{n (n-k) \log (n-k)}\right) \sigma_{k+1}(G)+3 \sqrt{n\sum_{j>k} \sigma_j^2(G)},
	\end{equation*}
	which holds with probability not less than $1-3(n-k)^{-(n-k)}$,
	and similarly
	\begin{equation*}
		\begin{aligned}
			\|B-\tilde{B}\| \le \|E_B\|
			&\le
			\left(1+6 \sqrt{(k+p) p \log p}\right) \sigma_{k+1}(U_1^{\mathrm{T}}B)+3 \sqrt{k+p} \sqrt{\sum_{j>k} \sigma_j^2(U_1^{\mathrm{T}}B)}
			\\&\le
			\left(1+6 \sqrt{(k+p) p \log p}\right) \sigma_{k+1}(B)+ 3 \sqrt{k+p} \sqrt{\sum_{j>k} \sigma_j^2(B)},
		\end{aligned}
	\end{equation*}
	with probability not less than $1-p^{-p}$,
	where we apply the result in \cite[Lemma 3.3.1]{horn1991PCtopics} that
		$\sigma_j(U_1^{\mathrm{T}}B)\le \sigma_j(B)$
	when $U_1$ is orthonormal.
	\end{proof}

	\begin{algorithm}[htb]
		\caption{Randomized RSVD algorithm}
		\label{Al-R-RSVD}
		\hspace*{0.02in} {\bf Input:}
		$A \in \mathbb{R}^{m \times n}$,
		$B \in \mathbb{R}^{m \times \ell}$,
		and $G \in \mathbb{R}^{d \times n}$, with
		with $\ell \ge d \ge m\ge n$, desired rank $k$, and the oversampling parameter $p$.
		\\
		\hspace*{0.02in} {\bf Output:}
		an RSVD of matrix triplet $(A,B,G)$, 
		$A=ZD_AW^{\mathrm{T}}$, $B\approx ZD_BU^{\mathrm{T}}$, $G\approx VD_GW^{\mathrm{T}}$.
		\begin{algorithmic}[1]
			\State Generate an $n\times n$ Gaussian random matrix $\Omega_1$.
			\State Form the $d\times n$ matrix $G\Omega_1$.
			\State Compute the $d\times n$ orthonormal matrix $H_1$ via the QR factorization
			$G\Omega_1 = H_1R$.
			\State Compute the GSVD of $(A,H_1^{\mathrm{T}}G)$:
			$\left[\begin{array}{c} A \\ H_1^{\mathrm{T}}G\end{array}\right]
			=\left[\begin{array}{ll}U_1 & \\ & \tilde{V}_1\end{array}\right]
			\left[\begin{array}{l}\Gamma_1 \\ \Sigma_1 \end{array}\right] Y_1^{\mathrm{T}}$.
			\State Form the $d\times n$ orthonormal matrix $V_1=H_1\tilde{V}_1$.
			\State Form the $m\times k_2$ matrix $G\Omega_2$.
			\State Compute the $(k_2+p)\times n$ orthonormal matrix $H_2$ via the QR factorization
			$(B^{\mathrm{T}}U_1)\Omega_2 = H_2R$.
			\State Compute the GSVD of
			$\left[H_2^{\mathrm{T}}\left(B^{\mathrm{T}}U_1\right),\Sigma_1^{-1}\Gamma_1^{\mathrm{T}}\right]$:
			$\left[\begin{array}{c}
				 H_2^{\mathrm{T}}\left(B^{\mathrm{T}}U_1\right) \\ \Sigma_1^{-1}\Gamma_1^{\mathrm{T}}
			 	   \end{array}\right]
			=\left[\begin{array}{ll}\tilde{U}_2 & \\ & V_2 \end{array}\right]
			\left[\begin{array}{l}\Gamma_2 \\ \Sigma_2 \end{array}\right] Y_2^{\mathrm{T}}$,
			where $\Sigma_2=\operatorname{diag}\left(\sigma_1, \ldots, \sigma_n\right)$.
			\State Form the $(k_2+p)\times k_2$ orthonormal matrix $U_2=H_2\tilde{U}_2$.
			\State Form the diagonal matrix
			$\Gamma_G=\mathrm{diag}(\gamma_1,\ldots, \gamma_n)$, $\gamma_i= \frac{\sigma_i}{\sqrt{\sigma_i^2+1}}$. 
			\State Form the orthonormal matrices
			$U=U_2\in\mathbb{R}^{(k_2+p)\times k_2}$,
			$V=V_1V_2\in\mathbb{R}^{d\times k_1}$,
			diagonal matrices $D_A=\Sigma_2^{\mathrm{T}}\Gamma_G$,
			$D_B=\Gamma_2^{\mathrm{T}}$, $D_G=\Gamma_G$,
			and the nonsingular matrices $Z=U_1Y_2\in\mathbb{R}^{m\times m}$,
			$W=Y_1\Sigma_1V_2\Gamma_G^{-1}\in\mathbb{R}^{n\times n}$.
		\end{algorithmic}
	\end{algorithm}
	\subsection{Randomization for L-DEIM Based RSVD-CUR}
	\hskip 2em
	Now we are ready to establish an efficient procedure for computing an approximate RSVD-CUR decomposition, along with a theoretical analysis of its error bound.
	Given a matrix triplet $(A,B,G)$, with
	$A\in\mathbb{R}^{m\times n}$ ,
	$B\in\mathbb{R}^{m\times l}$, and
	$G\in\mathbb{R}^{d\times n}$ ($\ell \ge d \ge m \ge n$) where $B$ and $G$ are of full rank.
	Our approach provides a rank-$k$ RSVD-CUR decomposition of the form (\ref{RSVD of (A,B,G)}), and the choice of indices $\mathbf{s}$, $\mathbf{s}_G$, $\mathbf{p}$, and $\mathbf{p}_B$ is guided by the orthonormal matrices and nonsingular matrices from the approximation of the rank-$\widehat{k}$ RSVD, where $\widehat{k}\le k$.
	The details are summarized in Algorithm \ref{Al-R-LDEIM-RSVD-CUR}.
	
	\hskip 2em
	The innovation of our approach has two aspects.
	First, we leverage the randomized algorithms (Algorithm \ref{Al-R-RSVD}) to accomplish the truncation procedure of the RSVD, where the random sampling technique can be used to identify a subspace that captures most of the action of a matrix,
	projecting a large-scale problem randomly to a smaller subspace that contains the main information.
	We then apply the deterministic algorithm to the associated small-scale problem,
	obtaining an approximate rank-$\widehat{k}$ RSVD of the form (\ref{R-RSVD}).
	Second,
	to further strengthen the efficiency of our algorithm scheme, we adopt the L-DEIM method for sampling instead of the DEIM.
	As described in Subsection \ref{subsection: Subset selection procedure}, compared to the DEIM scheme, the L-DEIM procedure is
	computationally more efficient and requires less than $k$ input vectors to select the indices.
	\begin{algorithm}[t!]
		\caption{L-DEIM based RSVD-CUR randomized algorithm}
		\label{Al-R-LDEIM-RSVD-CUR}
		\hspace*{0.02in} {\bf Input:}
		$A \in \mathbb{R}^{m \times n}$,
		$B \in \mathbb{R}^{m \times l}$,
		and
		$G \in \mathbb{R}^{d \times n}$
		with $l=d\ge m\ge n$, desired rank $k$, the \hspace*{0.02in} oversampling
		parameter $p$ and the specified parameter $\widehat{k}$.
		\\
		\hspace*{0.02in} {\bf Output:}
		A rank-$k$ RSVD-CUR decomposition\\
		\hspace*{0.02in}
		$A\approx A(:,\mathbf{p}) \cdot M_A \cdot A(\mathbf{s},:)$,
		$B\approx B(:,\mathbf{p}_B) \cdot M_B \cdot B(\mathbf{s},:)$,
		$G\approx G(:,\mathbf{p}) \cdot M_G \cdot G(\mathbf{s}_G,:)$.
		\begin{algorithmic}[1]
			\State Generate an $n\times n$ Gaussian random matrix $\Omega_1$.
			\State Form the $d\times n$ matrix $G\Omega_1$.
			\State Compute the $d\times n$ orthonormal matrix $H_1$ via the QR factorization
			$G\Omega_1 = H_1R_1$.
			\State Compute the GSVD of $(H_1^{\mathrm{T}}G,A)$:
			$\left[\begin{array}{c} H_1^{\mathrm{T}} G \\ A\end{array}\right]
			=\left[\begin{array}{ll}\tilde{V}_1 & \\ & U_1\end{array}\right]
			\left[\begin{array}{l}\Sigma_1 \\ \Gamma_1 \end{array}\right] Y_1^{\mathrm{T}}$.
			\State Form the $n\times n$ orthogonal matrix $V_1 = H_1\tilde{V}_1$.
			\State Generate an $m\times (\widehat{k}+p)$ Gaussian random matrix $\Omega_2$.
			\State Form the $l\times (\widehat{k}+p)$ matrix $(B^{\mathrm{T}}U_1)\Omega_2$.
			\State Compute the $l\times (\widehat{k}+p)$ orthonormal matrix $H_2$ via the QR factorization
			$(B^{\mathrm{T}}U_1)\Omega_2= H_2R_2$.
			\State Compute the GSVD of
			$(H_2^{\mathrm{T}}(B^{\mathrm{T}}U_1),\Sigma_1^{-1}\Gamma_1^{\mathrm{T}})$:
			$\left[\begin{array}{c}
				H_2^{\mathrm{T}} (B^{\mathrm{T}}U_1) \\ \Sigma_1^{-1}\Gamma_1^{\mathrm{T}}
			\end{array}\right]
			=\left[\begin{array}{ll}\tilde{U}_2 & \\ & V_2\end{array}\right]
			\left[\begin{array}{l}\Gamma_2 \\ \Sigma_2 \end{array}\right] Y_2^{\mathrm{T}}$,
			where $\Sigma_2=\operatorname{diag}\left(\sigma_1, \ldots, \sigma_n\right)$.
			\State Form the $l\times (\widehat{k}+p)$ orthonormal matrix $U_2 = H_2\tilde{U}_2$.
			\State Form the $n\times n$ diagonal matrix
			$\Gamma_G=\mathrm{diag}(\gamma_1,\ldots, \gamma_n)$, $\gamma_i=\frac{\sigma_i}{\sqrt{\sigma_i^2+1}}$.
			\State Form the orthonormal matrices
			$V=V_1V_2$, $U=U_2$ and nonsingular matrices $Z=U_1Y_2$ and $W=Y_1\Sigma_1V_2\Gamma_G^{-1}$.
			\State $\mathbf{p}_B=\mathrm{l\mbox{-}deim}(U)$.
			\State $\mathbf{p}=\mathrm{l\mbox{-}deim}(W)$.
			\State $\mathbf{s}=\mathrm{l\mbox{-}deim}(Z)$.
			\State $\mathbf{s}_G=\mathrm{l\mbox{-}deim}(V)$.
			%
			\State
			Compute
			$M_A=A(:, \mathbf{p}) \backslash\left(A / A\left(\mathbf{s}_A,:\right)\right)$,
			$M_B=B(:, \mathbf{p}) \backslash\left(B / B\left(\mathbf{s}_B,:\right)\right)$.
		\end{algorithmic}
	\end{algorithm}

	\hskip 2em
	We provide a rough error analysis that shows that the accuracy of the proposed
	algorithm is closely associated with the error of the approximation RSVD.
	The analysis follows the results in \cite{sorensen2016SIAMdeim,gidisu2022SIAMgeneralized,gidisu2022ARXIVrsvd} with some necessary modifications.
	We begin by partitioning the matrices in (\ref{R-RSVD})
	\begin{equation*}
		\begin{aligned}
			U &=\left[\begin{array}{ll}
				U_{\widehat{k}} & \widehat{U}
			\end{array}\right], \quad V=\left[\begin{array}{ll}
				V_{\widehat{k}} & \widehat{V}
			\end{array}\right], \quad W=\left[\begin{array}{ll}
				W_{\widehat{k}} & \widehat{W}
			\end{array}\right], \quad Z=\left[\begin{array}{ll}
				Z_{\widehat{k}} & \widehat{Z}
			\end{array}\right], \\
			D_A &=\operatorname{diag}\left(D_{A_{\widehat{k}}}, \widehat{D}_A\right), \quad
			D_B=\operatorname{diag}\left(D_{B_{\widehat{k}}}, \widehat{D}_B\right), \quad
			D_G=\operatorname{diag}\left(D_{G_{\widehat{k}}}, \widehat{D}_G\right),
		\end{aligned}
	\end{equation*}
	where
	$\widehat{D}_A \in \mathbb{R}^{(m-\widehat{k}) \times(n-\widehat{k})}$,
	$\widehat{D}_B \in \mathbb{R}^{(m-\widehat{k}) \times(l-\widehat{k})}$, and
	$\widehat{D}_G \in \mathbb{R}^{(d-\widehat{k}) \times(n-\widehat{k})}$.
	As with the DEIM-type GCUR method in \cite{gidisu2022SIAMgeneralized}, the lack of orthogonality of the basis vectors in $W$ and $Z$ from the RSVD
	necessitates some additional work.
	Mimicking the techniques in \cite{gidisu2022ARXIVrsvd},
	here we take a QR factorization of $W$ and $Z$ to obtain an orthonormal basis to facilitate the analysis,
	\begin{equation*}
		\begin{aligned}
			{\left[\begin{array}{ll}
					Z_{\widehat{k}} & \widehat{Z}
				\end{array}\right]=Z=Q_Z T_Z=\left[\begin{array}{ll}
					Q_{Z_{\widehat{k}}} & \widehat{Q}_Z
				\end{array}\right]\left[\begin{array}{cc}
					T_{Z_{\widehat{k}}} & T_{Z_{12}} \\
					0 & T_{Z_{22}}
				\end{array}\right]=\left[\begin{array}{ll}
					Q_{Z_{\widehat{k}}} T_{Z_{\widehat{k}}} & Q_Z \widehat{T}_Z
				\end{array}\right],} \\
			{\left[\begin{array}{ll}
					W_{\widehat{k}} & \widehat{W}
				\end{array}\right]=W=Q_W T_W=\left[\begin{array}{ll}
					Q_{W_{\widehat{k}}} & \widehat{Q}_W
				\end{array}\right]\left[\begin{array}{cc}
					T_{W_{\widehat{k}}} & T_{W_{12}} \\
					0 & T_{W_{22}}
				\end{array}\right]=\left[\begin{array}{ll}
					Q_{W_{\widehat{k}}} T_{W_{\widehat{k}}} & Q_W \widehat{T}_W
				\end{array}\right],}
		\end{aligned}
	\end{equation*}
	where we have denoted
	\begin{equation*}
		\widehat{T}_Z:=\left[\begin{array}{c}
			T_{Z_{12}} \\
			T_{Z_{22}}
		\end{array}\right],
		\quad
		\widehat{T}_W:=\left[\begin{array}{l}
			T_{W_{12}} \\
			T_{W_{22}}
		\end{array}\right].
	\end{equation*}
	It is straightforward to check that
	\begin{equation*}
		\begin{aligned}
			B &
			=Z_{\widehat{k}} D_{B_k} U_{\widehat{k}}^{\mathrm{T}} + \widehat{Z} \widehat{D}_B \widehat{U}^{\mathrm{T}} + E_B
			=Q_{Z_{\widehat{k}}} T_{Z_{\widehat{k}}} D_{B_{\widehat{k}}} U_{\widehat{k}}^{\mathrm{T}}+Q_Z \widehat{T}_Z \widehat{D}_B \widehat{U}^{\mathrm{T}} + E_B,
			\\
			G &
			=V_{\widehat{k}} D_{G_{\widehat{k}}} W_{\widehat{k}}^{\mathrm{T}}+\widehat{V} \widehat{D}_G \widehat{W}^{\mathrm{T}} + E_G
			=V_{\widehat{k}} D_{G_{\widehat{k}}} T_{W_{\widehat{k}}}^{\mathrm{T}} Q_{W_{\widehat{k}}}^{\mathrm{T}}+V_{\widehat{k}} \widehat{D}_G \widehat{T}_W^{\mathrm{T}} Q_W^{\mathrm{T}} + E_G,
		\end{aligned}
	\end{equation*}
	where $E_B$ and $E_G$ satisfy the probabilistic error bounds (\ref{Err of R-RSVD-B}) and (\ref{Err of R-RSVD-G}).
	Since Algorithm \ref{Al-R-RSVD} provides an exact decomposition of $A,$ the error bound for $A$ in \cite[Proposition 2]{gidisu2022ARXIVrsvd}
	\begin{equation}\label{Err of R-LDEIM-RSVD-CUR-A}
		\|A-C_A M_A R_A\|
		\le
		\alpha_{k+1} \cdot
		\left(\sqrt{\frac{n {\widehat{k}}}{3}} 2^{\widehat{k}}+\sqrt{\frac{m {\widehat{k}}}{3}} 2^{\widehat{k}}\right) \cdot
		\left\|\widehat{T}_W\right\|\left\|\widehat{T}_Z\right\|,
	\end{equation}
	still holds. Here $\alpha_{k+1}$ is the $(k+1)$th diagonal entry of $D_A$, which is ordered non-increasingly.
	The following theorem roughly quantifies the error bounds for
	$\|B-C_B M_B R_B\|$ and $\|G-C_G M_G R_G\|.$
	\begin{theorem}\label{Th err of R-LDEIM-RSVD-CUR}
		Suppose that a rank-$k$ RSVD-CUR decomposition for $(A,B,G)$ of the form (\ref{RSVD-CUR of (A,B,G)}) is produced by Algorithm \ref{Al-R-LDEIM-RSVD-CUR}, where
		$S=I(:,\mathbf{s})$, $S_G=I(:,\mathbf{s}_G)$, $P=I(:,\mathbf{p})$ and $P_B=I(:,\mathbf{p}_B)$ are the index selection matrices,
		and $p$ is the oversampling parameter. Let $\eta_G = \sqrt{\frac{n {\widehat{k}}}{3}} 2^{\widehat{k}} + \sqrt{\frac{d {\widehat{k}}}{3}} 2^{\widehat{k}}$,
        and $\eta_B = \sqrt{\frac{l {\widehat{k}}}{3}} 2^{\widehat{k}} + \sqrt{\frac{m {\widehat{k}}}{3}} 2^{\widehat{k}}$. Then
				\begin{equation*}
						\|G-C_G M_G R_G\|
						\le
						\eta_G \cdot
						\left(
					\|E_G\| 
						+
						\left\|\widehat{T}_W\right\|
						\right),
                       ~~~
                       \|B-C_B M_B R_B\|
						\le
						\eta_B \cdot
						\left(
					\|E_B \| 
						+
						\left\|\widehat{T}_Z\right\|
						\right),
				\end{equation*}
		where
   \begin{equation*}
		\|E_G\| \le \left(1+6 \sqrt{n(n-{\widehat{k}}) \log (n-\widehat{k})}\right) \sigma_{\widehat{k}+1}(G) +3 \sqrt{n \sum_{j>\widehat{k}} \sigma_j^2(G)},
	\end{equation*}
   \begin{equation*}
		\|E_B \| \leq	\left(1+6 \sqrt{(\widehat{k}+p) p \log p}\right) \sigma_{\widehat{k}+1}(B)+ 3  \sqrt{(\widehat{k}+p )\sum_{j>\widehat{k}} \sigma_j^2(B)}
	\end{equation*}
   which hold with probability not less than $1-(n-\widehat{k})^{-(n-\widehat{k})}$ and $1-3 p^{-p},$ respectively.
	
	\end{theorem}
	\begin{proof}
		It suffices to prove the bound for
		$\|G-C_G M_G R_G\|.$
	From the definition of $M_G$, using the result in \cite{mahoney2009PNAScur},
	we have
		\begin{equation}
			G-C_G M_G R_G
			=G-C_G C_G^{+} G R_G^{\dagger} R_G
			=(I-C_G C_G^{\dagger})G + C_G C_G^{\dagger} G (I-R_G^{\dagger} R_G).
		\end{equation}
        Then $$\left\|G-C_G M_G R_G\right\| \le
				\|(I-C_G C_G^{\dagger}) G\|
				+\|G (I-R_G^{\dagger} R_G)\|.$$
		Given index selection matrix $P$ from the L-DEIM scheme on matrix $W_{\widehat{k}}$,
		and suppose that $Q_{W_{\widehat{k}}}$
		is an orthonormal basis for $\mathrm{Ran}(W_{\widehat{k}})$.
		We form 
		$\mathbb{P}=P(Q_{W_{\widehat{k}}}^{\mathrm{T}} P)^{\dagger} Q_{W_{\widehat{k}}}^{\mathrm{T}}$:
		an oblique projector with
		$P(W_{\widehat{k}}^{\mathrm{T}} P)^{\dagger} W_{\widehat{k}}^{\mathrm{T}}
		=P(Q_{W_{\widehat{k}}}^{\mathrm{T}} P)^{\dagger} Q_{W_{\widehat{k}}}^{\mathrm{T}}$
		(\cite[Equation 3.6]{chaturantabut2010SIAMnonlinear})
		and
		we also have
		$Q_{W_{\widehat{k}}}^{\mathrm{T}} \mathbb{P}
		=Q_{W_{\widehat{k}}}^{\mathrm{T}} P(Q_{W_{\widehat{k}}}^{\mathrm{T}} P)^{\dagger} Q_{W_{\widehat{k}}}^{\mathrm{T}}=Q_{W_{\widehat{k}}}^{\mathrm{T}}$,
		which implies $Q_{W_{\widehat{k}}}^{\mathrm{T}}(I-\mathbb{P})=0$.
		From \cite[Lemmas 2 and 3]{hendryx2021MLextended}, we obtain that
		\begin{equation*}
			\|(I-C_G C_G^{\dagger}) G\| \leq \|G(I-{\mathbb{P}})\| =\|G(I-Q_{W_{\widehat{k}}} Q_{W_{\widehat{k}}}^{\mathrm{T}})(I-\mathbb{P})\|
                                         \leq \|G(I-Q_{W_{\widehat{k}}} Q_{W_{\widehat{k}}}^{\mathrm{T}})\|\|I-{\mathbb{P}}\|,
		\end{equation*}
        \begin{equation*}
					\|G(I-R_G^{\dagger} R_G)\| \leq  \|(I-{\mathbb{S}}) G\| =\|(I-{\mathbb{S}})(I-V_{\widehat{k}} V_{\widehat{k}}^{\mathrm{T}}) G\|
                                               \leq \|(I-{\mathbb{S}})\|   \|(I-V_{\widehat{k}} V_{\widehat{k}}^{\mathrm{T}}) G \|.
        \end{equation*}
		Since $\widehat{k}<r$, ${\mathbb{P}} \neq 0, {\mathbb{P}} \neq I$ and $\mathbb{S}\neq 0, \mathbb{S}\neq I.$
        By \cite[Lemma 4.1]{szyld2006NMmany}, we have
		\begin{equation*}
			\|I-\mathbb{P}\|
			=\|\mathbb{P}\|
			=\|(Q_{W_{\widehat{k}}}^{\mathrm{T}} P)^{\dagger}\|,~~~
              \|I-\mathbb{S}\|=\|\mathbb{S}\|=\|(S^{\mathrm{T}} V_{\widehat{k}})^{\dagger}\|.
		\end{equation*}
		Using the partitioning of $G$, we have
	    \begin{equation*}
	    	\begin{aligned}
	    		G Q_{W_{\widehat{k}}} Q_{W_{\widehat{k}}}^{\mathrm{T}} &
	    		=\left[\begin{array}{ll}
	    			V_{\widehat{k}} & \widehat{V}
	    		\end{array}\right]
    			\left[\begin{array}{cc}
	    			D_{G_{\widehat{k}}} & 0 \\
	    			0 & \widehat{D}_G
	    		\end{array}\right]
    			\left[\begin{array}{cc}
	    			T_{W_{\widehat{k}}}^{\mathrm{T}} & 0 \\
	    			T_{W_{12}}^{\mathrm{T}} & T_{W_{22}}^{\mathrm{T}}
	    		\end{array}\right]
    			\left[\begin{array}{c}
	    			I_{\widehat{k}} \\
	    			0
	    		\end{array}\right]
    			Q_{W_{\widehat{k}}}^{\mathrm{T}}
    			+ E_G Q_{W_{\widehat{k}}} Q_{W_{\widehat{k}}}^{\mathrm{T}}
    			\\
	    		&=
	    		V_{\widehat{k}} D_{G_{\widehat{k}}} T_{W_{\widehat{k}}}^{\mathrm{T}} Q_{W_{\widehat{k}}}^{\mathrm{T}}
	    		+\widehat{V} \widehat{D}_G T_{W_{12}}^{\mathrm{T}} Q_{W_{\widehat{k}}}^{\mathrm{T}}
	    		+ E_G Q_{W_{\widehat{k}}} Q_{W_{\widehat{k}}}^{\mathrm{T}},
	    	\end{aligned}
	    \end{equation*}
    and hence
    \begin{equation*}
    	\begin{aligned}
    		G(I-Q_{W_{\widehat{k}}} Q_{W_k}^{\mathrm{T}})
    		&=\widehat{V} \widehat{D}_G \widehat{T}^{\mathrm{T}} Q^{\mathrm{T}}
    		-\widehat{V} \widehat{D}_G T_{W_{12}}^{\mathrm{T}} Q_{W_{\widehat{k}}}^{\mathrm{T}}
    		-E_G Q_{W_{\widehat{k}}} Q_{W_{\widehat{k}}}^{\mathrm{T}}
            =\widehat{V} \widehat{D}_G T_{W_{22}}^{\mathrm{T}} \widehat{Q}_W^{\mathrm{T}}
    		 - E_G Q_{W_{\widehat{k}}} Q_{W_{\widehat{k}}}^{\mathrm{T}}.
    	\end{aligned}
    \end{equation*}
	This implies
	\begin{equation*}
		\|G(I-Q_{W_{\widehat{k}}} Q_{W_{\widehat{k}}}^{\mathrm{T}})\|
		\leq
		\gamma_{\widehat{k}+1} \left\|T_{W_{22}}\right\| + \|E_G\|
		\leq
		\left\|T_{W_{22}}\right\| + \|E_G\|
	\end{equation*}
	and then
	\begin{equation*}
		\|(I-C_G C_G^{\dagger}) G\|
		\leq
		\|G(I-Q_{W_{\widehat{k}}} Q_{W_{\widehat{k}}}^{\mathrm{T}})\|\|I-{\mathbb{P}}\|
		\leq
		\|(Q_{W_{\widehat{k}}}^{\mathrm{T}} P)^{\dagger}\|
		\cdot
		( \left\|T_{W_{22}}\right\| + \|E_G\| ),
	\end{equation*}
   Similarly, we have
	\begin{equation*}
		\|G(I-R_G^{\dagger} R_G)
		\le
		\|(S^{\mathrm{T}} V_{\widehat{k}})^{\dagger}\| \cdot ( \|\widehat{T}_W\|+\|E_G\|)
		\le
		\|(S^{\mathrm{T}} V_{\widehat{k}})^{\dagger}\| \cdot ( \|\widehat{T}_W\|+\|E_G\|).
	\end{equation*}
	
	Then it follows that
	\begin{equation*}\label{Proof of R-LDEIM-RSVD-CUR eq1}
		\begin{aligned}
			\|G-C_G M_G R_G\|
			\le &
			\left( \|(Q_{W_{\widehat{k}}}^{\mathrm{T}} P)^{\dagger}\|+ \|(S^{\mathrm{T}} V_{\widehat{k}})^{\dagger}\|\right)
			\cdot
			(\|\widehat{T}_W\|+\|E_G\|).
		\end{aligned}
	\end{equation*}
    Using the upper bounds \cite{gidisu2022ARXIVrsvd}
	\begin{equation*}
		\begin{aligned}
			&\|(Q^{\mathrm{T}}_{W_{\widehat{k}}}P)^{\dagger}\| < \sqrt{\frac{n\widehat{k}}{3}}2^{\widehat{k}},
			\quad
			&\|(S_G^{\mathrm{T}} V_{\widehat{k}})^{\dagger}\| < \sqrt{\frac{d{\widehat{k}}}{3}}2^{\widehat{k}}.
		\end{aligned}
	\end{equation*}
   and applying the probabilistic error bound (\ref{Err of RGSVD}),
   we obtain the desired result.
	\end{proof}

	\hskip 2em
	Comparing the results of the error bounds in Theorem \ref{Th err of R-LDEIM-RSVD-CUR} to \cite[Theorem 4.3]{gidisu2022ARXIVrsvd} that
	\begin{equation*}
		\|G-C_G M_G R_G\|
		\le
		\gamma_{k+1}
		\cdot
		\eta_G
		\cdot
		\left\|\widehat{T}_W\right\|,
		\|B-C_B M_B R_B\|
		\le
		\eta_B
		\cdot
		\left\|\widehat{T}_Z\right\|,
	\end{equation*}
	our results involve the item
	$(1+6 \sqrt{n(n-\widehat{k}) \log (n-\widehat{k})}) \sigma_{\widehat{k}+1}(G)+3 \sqrt{n \sum_{j>\widehat{k}} \sigma_j^2(G)}$
	in the error bound of $\| G-C_G M_G R_G\|$
	and the item
	$(1+6 \sqrt{(\widehat{k}+p) p \log p}) \sigma_{\widehat{k}+1}(B)+3 \sqrt{\widehat{k}+p} \sqrt{\sum_{j>\widehat{k}} \sigma_j^2(B)}$
	in the error bound of $\|B-C_B M_B R_B\|$, respectively. 
	Therefore, our randomized algorithm works well for the matrices whose singular
	values exhibit some decay.

	\section{Numerical Examples}
	\hskip 2em
	In this section, we check the accuracy and the computational cost
	of our algorithms on several synthetic and real-world datasets.
	All computations are carried out in MATLAB R2020a on a
	computer with an AMD Ryzen 5 processor and 16 GB RAM.
	To facilitate the comparison between different algorithms, we define the following acronyms.
	
	\hskip 2em 1. DEIM-GCUR$-$
	implements the GCUR algorithm with column subset selection implemented using the DEIM algorithm (Algorithm \ref{Al-DEIM}) labeled ``DEIM-GCUR'', as summarized in \cite[Algorithm 4.1]{gidisu2022SIAMgeneralized}).
	
	\hskip 2em 2. R-GCUR $-$
	applies the randomized GCUR algorithm with column subset selection implemented using either the DEIM algorithm labeled ``R-DEIM-GCUR'',
	summarized in Algorithm \ref{Al-R-DEIM-GCUR},
	or the L-DEIM algorithm (Algorithm \ref{Al-LDEIM}) labeled ``R-LDEIM-GCUR'' as summarized in Algorithm \ref{Al-R-LDEIM-GCUR}.
	
	\hskip 2em 3. RSVD-CUR $-$
	implements the RSVD-CUR decomposition algorithm by using the DEIM labeled ``DEIM-RSVD-CUR'',
	as summarized in \cite[Algorithm 3]{gidisu2022ARXIVrsvd},
	or the L-DEIM algorithm, labeled ``LDEIM-RSVD-CUR'', as summarized in \cite[Algorithm 4]{gidisu2022ARXIVrsvd}.
	
	\hskip 2em 4. R-LDEIM-RSVD-CUR $-$
	implements the randomized RSVD-CUR algorithm based on the L-DEIM procedure labeled ``R-LDEIM-RSVD-CUR'' (Algorithm \ref{Al-R-LDEIM-RSVD-CUR}) to produce the RSVD-CUR decomposition.

	$\mathbf{Example}$ $\mathbf{5.1}$
	This experiment is a variation of prior experiments in \cite[Section 3.4.4]{hansen1998SIAMrank}, \cite[Experiment 5.1]{gidisu2022SIAMgeneralized} and \cite[Example 6.1]{sorensen2016SIAMdeim}.
	In this experiment, we examine the performance of the GCUR, R-GCUR algorithms,
	and the CUR decomposition in the context of matrix recovery of the original matrix $A$ from $A_E= A + E$, where $E$ is a noise matrix.
	First, we construct a matrix $A\in\mathbb{R}^{m\times n}$ of the form
	\begin{equation*}
		A=\sum_{j=1}^{10} \frac{2}{j} \mathbf{x}_j \mathbf{y}_j^{\mathrm{T}}+\sum_{j=11}^{50} \frac{1}{j} \mathbf{x}_j \mathbf{y}_j^{\mathrm{T}},
	\end{equation*}
	where
	$\mathbf{x}_j \in \mathbb{R}^{m}$ and $\mathbf{y}_j \in \mathbb{R}^{n}$
	are sparse vectors with random nonnegative entries, and alternatively, in MATLAB,
	$\mathbf{x}_j=\mathtt{sprand}(m,1,0.025)$ and $\mathbf{y}_j=\mathtt{sprand}(n,1,0.025)$.
	Just as in \cite{gidisu2022SIAMgeneralized}, we construct a correlated Gaussian noise $E$ whose entries have zero mean and a Toeplitz covariance structure, i.e.,
	in MATLAB
	$E = \varepsilon \frac{\|F\|}{\|A\|}F$,
	desired-cov($F$)=$\mathtt{toeplitz}$ $(0.99^0, \ldots, 0.99^{n-1})$,
	$B=\mathtt{chol}$({{desired}-{cov}}$(F)$),
	and $F=\mathtt{randn}(m, n) \cdot B$.
	$\varepsilon$ represents the noise level and $\varepsilon \in\{0.05,0.1,0.15,0.2\}$.
	The performance is assessed based on the 2-norm of the relative matrix approximation error, i.e.,
		$\mathrm{Err}={\|A-\widetilde{A}\|}/{\|A\|},$
	where $\widetilde{A}$ is the approximated low-rank matrix.
	
	\hskip 2em
	We first compare the accuracy of the
	GCUR algorithms with their randomized counterparts R-GCUR and the standard DEIM-CUR decomposition for reconstructing the low-rank matrix $A$ for different noise levels.
	As inputs, we fix $m=10000$, $n=300$ and using the target rank $k$ varies from $1$ to $50$, and the parameter contained in the L-DEIM procedure is $\widehat{k}=k/2$.
	The relative errors are plotted in Figure \ref{Fig-Exp1}.
	We observe that the GCUR and R-GCUR techniques achieve a comparable relative error.
	Consistent with the results in \cite{gidisu2022SIAMgeneralized}, the R-GCUR algorithm performs significantly well under high noise.
	Besides, we observe that, as $k$ approaches $\mathrm{rank}(A)$, however, the relative errors of both the GCUR and the R-GCUR do not decrease any more.
	\cite{gidisu2022SIAMgeneralized} attributes this phenomenon to the fact that the relative error is saturated by the noise, considering we pick the columns and rows of the noisy data.
	
	\hskip 2em
	The analysis of the proposed algorithms implies that our randomized algorithms are less expensive compared to their deterministic counterparts.
	To illustrate this, we record the running time in seconds (denoted as CPU) and the approximation quality Err of
	the GCUR and R-GCUR for reconstructing matrix $A$ for different noise levels $\varepsilon=0.2, 0.1, 0.05$ as the dimension and the target rank $k$ increase.
	According to the conclusions summarized in \cite{gidisu2022Arxivhybrid}, the L-DEIM procedure may be comparable to the original DEIM method when the target rank $k$ is at most twice the available $\widehat{k}$ singular vectors.
	Therefore, here we set the parameter $\widehat{k}$ contained in the L-DEIM to be $\widehat{k}=k/2$, and the oversampling parameter $p=5$.
	We record the results in Tables \ref{Table1-Exp1}-\ref{Table3-Exp1}.
	It is clear from the running time that the algorithms R-DEIM-GCUR and R-LDEIM-GCUR
	have a huge advantage in computing speed over the non-random GCUR method, and the R-LDEIM-GCUR achieves the smallest running time among the three sets of experiments.
	 \begin{figure}[t!]
		\centering
		\subfigure[$\varepsilon=0.2$]
		{\includegraphics[width=0.46\textwidth,height=0.3\textwidth]{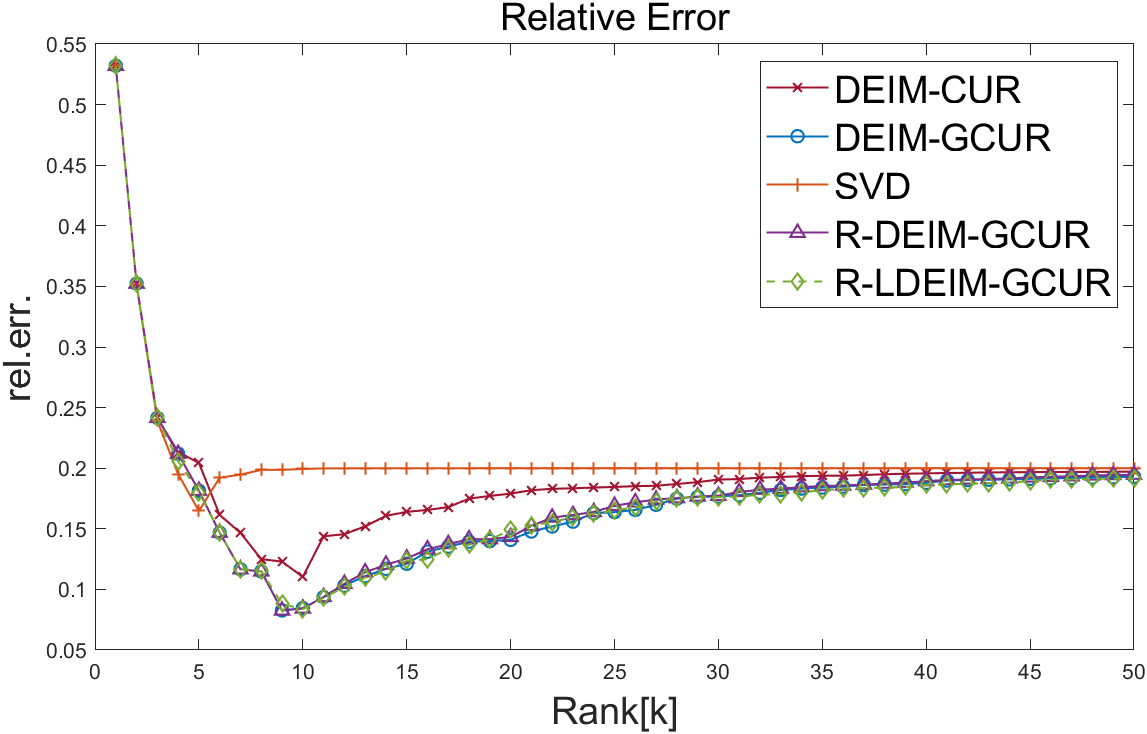}	}
		\subfigure[$\varepsilon=0.15$]
		{\includegraphics[width=0.46\textwidth,height=0.3\textwidth]{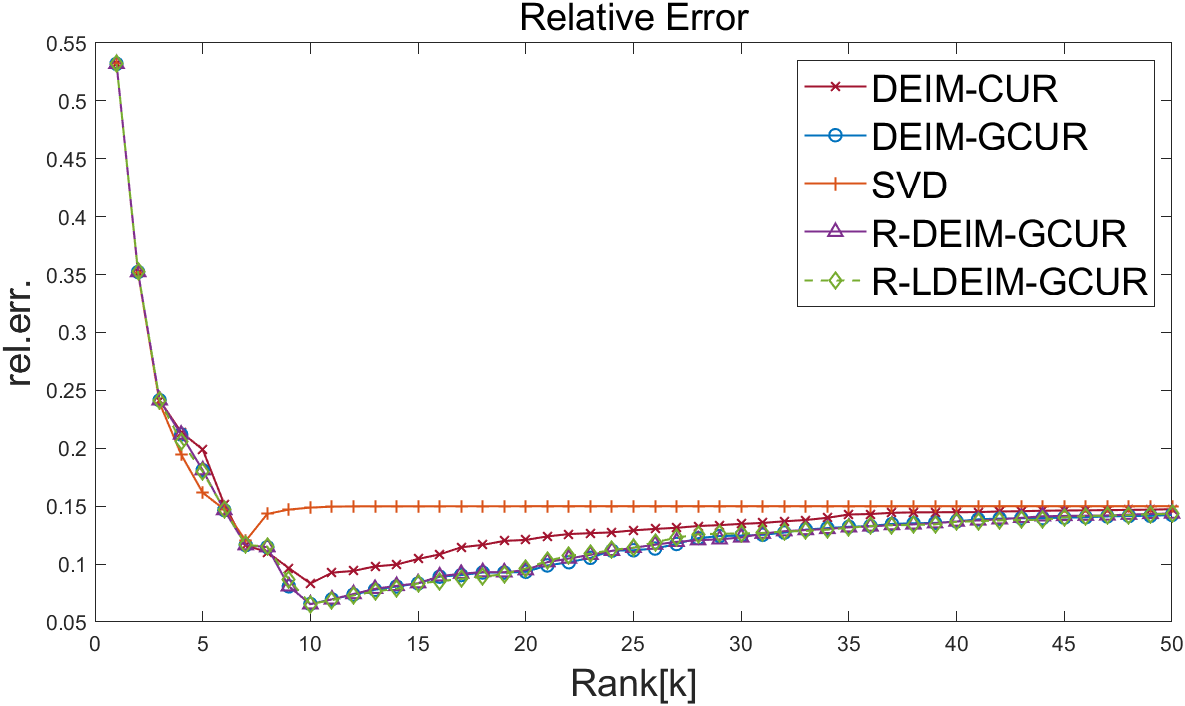}	}
		\subfigure[$\varepsilon=0.1$]
		{\includegraphics[width=0.46\textwidth,height=0.3\textwidth]{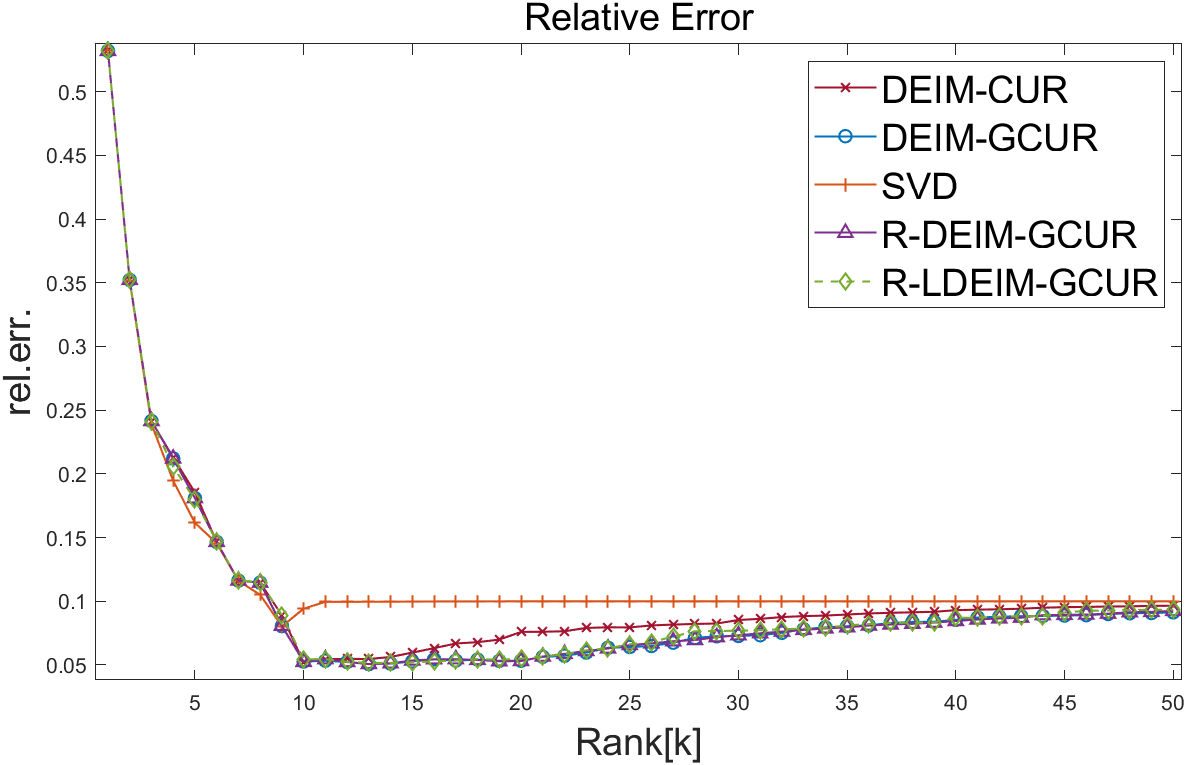}	}
		\subfigure[$\varepsilon=0.05$]
		{\includegraphics[width=0.46\textwidth,height=0.3\textwidth]{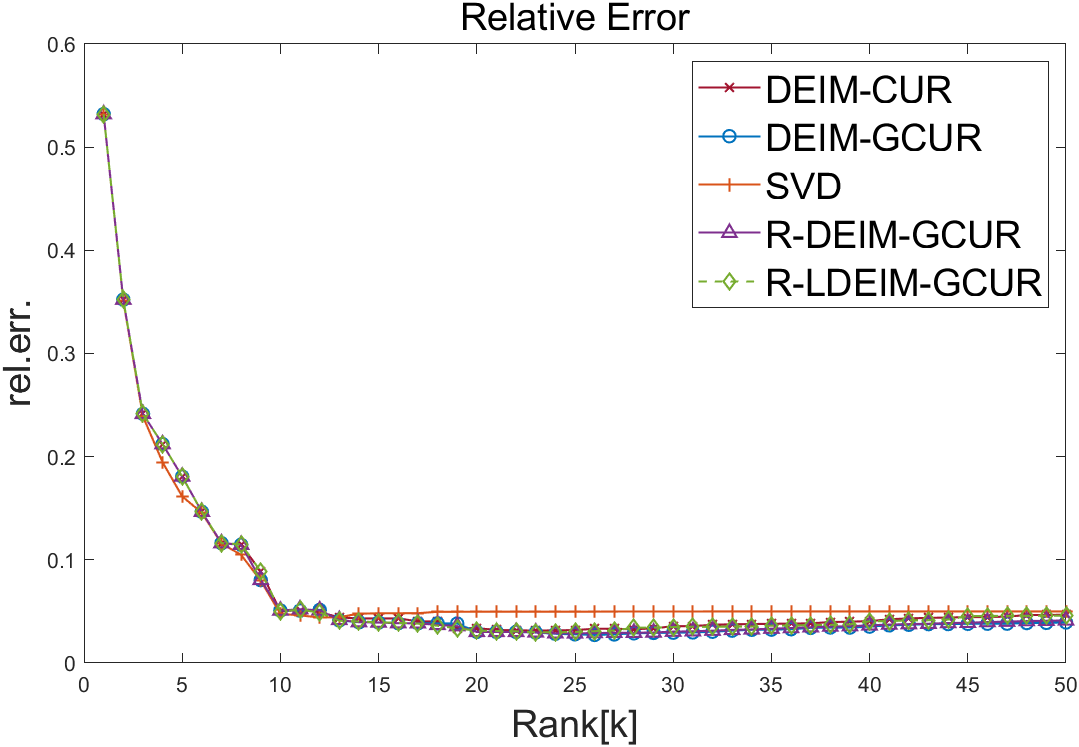}	}
				\caption{
				This figure depicts the accuracy comparison of R-GCUR approximations against the DEIM-CUR approximation and the GCUR decomposition in the recovery matrix $A$. 
				The relative errors are shown as a function of rank $k$ for varying values of $\varepsilon$: $0.2$, $0.15$, 0.1, and $0.05$.
			}
		\label{Fig-Exp1}
	\end{figure}

	\begin{table}[H]
		\caption{Comparison of GCUR and randomized algorithms (R-DEIM-GCUR and R-LDEIM-GCUR) in CPU and relative error as the dimension and the target rank $k$ increase, with noise level $\varepsilon=0.2$. }
		\label{Table1-Exp1}
		\setlength{\tabcolsep}{0.7mm}
		\centering
		\begin{tabular}{c c c c c c }
			\begin{tabular}{|c|c|c|c|c|c|}
				\hline
				\multicolumn{2}{|c|}{$(m,n,k)$ }
				& $(10000,200,20)$ & $(50000,200,20)$ &$(100000,500,30)$ & $(200000,1000,40)$\\
				\hline
				\multirow{2}{*}{ GCUR }
				& Err
				& $0.15725$ & $0.14842$ & $0.18058$ & $0.17292$ \\
				\cline { 2 - 6 }
				& CPU
				& $0.10197$ & $0.54697$ & $5.4971$ & $40.001$ \\
				\hline \multirow{2}{*}{ R-DEIM-GCUR }
				& Err
				& $0.14524$ & $0.16584$ & $0.18260$ & $0.17772$ \\
				\cline { 2 - 6 }
				& CPU
				& $0.027867$ & $0.11137$ & $0.49037$ & $2.4217$ \\
				\hline \multirow{2}{*}{ R-LDEIM-GCUR }
				& Err
				& $0.16173$ & $0.14640$ & $0.16955$ & $0.16758$ \\
				\cline { 2 - 6 }
				& CPU
				& $0.018809$ & $0.056930$ & $0.28019$ & $1.5563$ \\
				\hline
			\end{tabular}
		\end{tabular}	
	\end{table}
	\begin{table}[H]
		\caption{Comparison of GCUR and randomized algorithms (R-DEIM-GCUR and R-LDEIM-GCUR) in CPU and relative error as the dimension and the target rank $k$ increase, with noise level $\varepsilon=0.1$.}
		\label{Table2-Exp1}
		\setlength{\tabcolsep}{0.5mm}
		\centering
		\begin{tabular}{c c c c c c}
			\begin{tabular}{|c|c|c|c|c|c|}
				\hline
				\multicolumn{2}{|c|}{$(m,n,k)$ }
				& $(10000,1000,30)$ & $(100000,500,30)$ &$(200000,1000,40)$ & $(200000,1000,50)$ \\
				\hline
				\multirow{2}{*}{ GCUR } & Err
				& $0.16493$ & $0.18058$ & $0.17292$ & $0.18699$\\
				\cline { 2 - 6 } & CPU
				& $2.1302$ & $4.5977$ & $33.783$ & $51.365$ \\
				\hline \multirow{2}{*}{ R-DEIM-GCUR } & Err
				& $0.16524$ & $0.18260$ & $0.17772$ & $0.18614$\\
				\cline { 2 - 6 } & CPU
				& $0.56099$ & $0.47876$ & $2.0406$ & $3.7259$ \\
				\hline \multirow{2}{*}{ R-LDEIM-GCUR } & Err
				& $0.16906$ & $0.16955$ & $0.16758$ & $0.172631$\\
				\cline { 2 - 6 } & CPU
				& $0.50487$ & $0.27769$ & $1.2856$ & $1.5229$\\
				\hline
			\end{tabular}
		\end{tabular}	
	\end{table}
	\begin{table}[H]
		\caption{Comparison of GCUR and randomized algorithms (R-DEIM-GCUR and R-LDEIM-GCUR) in CPU and relative error as the dimension and the target rank $k$ increase, with noise level $\varepsilon=0.05$.}
		\label{Table3-Exp1}
		\setlength{\tabcolsep}{0.5mm}
		\centering
		\begin{tabular}{c c c c c c}
			\begin{tabular}{|c|c|c|c|c|c|}
				\hline
				\multicolumn{2}{|c|}{$(m,n,k)$ }
				& $(10000,500,20)$ & $(100000,500,30)$ &$(150000,1000,40)$ & $(200000,1000,50)$ \\
				\hline
				\multirow{2}{*}{ GCUR } & Err
				& $0.13828$ & $0.18058$ & $0.18230$ & $0.18699$ \\
				\cline { 2 - 6 } & CPU
				& $0.56859$ & $4.5496$ & $25.523$ & $32.360$ \\
				\hline \multirow{2}{*}{ R-DEIM-GCUR } & Err
				& $0.13089$ & $0.18260$ & $0.17513$ & $0.18614$ \\
				\cline { 2 - 6 } & CPU
				& $0.10336$ & $0.48947$ & $1.8595$ & $2.6152$  \\
				\hline \multirow{2}{*}{ R-LDEIM-GCUR } & Err
				& $0.13807$ & $0.16955$ & $0.17975$ & $0.17263$ \\
				\cline { 2 - 6 } & CPU
				& $0.079551$ & $0.28887$ & $1.0581$ & $1.4994$ \\
				\hline
			\end{tabular}
		\end{tabular}	
	\end{table}

	$\mathbf{Example}$ $\mathbf{5.2}$
	This experiment evaluates the performance of our randomized algorithms using synthetic data sets.
	These data sets are generated based on the procedures described in \cite[Example 5.3]{gidisu2022SIAMgeneralized} and \cite{abid2018Natureexploring},
	which provide valuable insights into scenarios where the CUR and GCUR techniques effectively address the issue of subgroups.
	Our task is to reduce the dimension of target data set $A$, which contains $4m$ data points in a $3d$-dimensional feature space with four different subgroups.
	Each of these subgroups has distinct variances and means, and their detailed characteristics are summarized in Figure \ref{groups}.
	\begin{figure}[htbp]
		\centering
		\includegraphics[width=0.7\textwidth,height=0.3\textwidth]{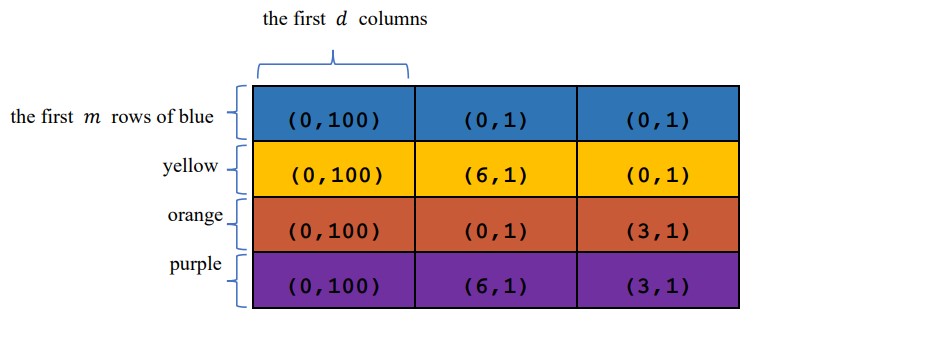}	
		\caption{
			The $4m\times 3d$ data set $A$ contains four subgroups, and each group has $m$ data points with different mean and variance.
			The blocks labeled $(a,b)$ denote that it is sampled from a normal
			distribution with a mean of $a$ and a variance of $b$.
		}
		\label{groups}
	\end{figure}
	Usually, this task can be accomplished by the principal component analysis based on SVD, but this method fails due to the variation along the first $d$ columns of the target data set is significantly larger than in any other direction.
	\cite[Example 5.3]{gidisu2022SIAMgeneralized} demonstrates that contrastive principal component analysis (cPCA)\cite{abid2018Natureexploring} can effectively solve this problem,
	which identifies low-dimensional space that is enriched in a dataset relative to comparison data.
	Specifically, following the operations in \cite{gidisu2022SIAMgeneralized}, we construct the background data set $B$, which is drawn from a normal distribution,
	and the variance and mean of the first $d$ columns, columns $d+1$ to $2d$, and the last $d$ columns of $B$ are $100,9,1$ and $0,0,0$, respectively.
	Then we can extract characteristics for clustering the subgroups in $A$ by optimizing the variance of $A$ while minimizing that of $B$, which leads to a trace ratio maximization problem \cite{chen2018IEEEnonlinear}
	\begin{equation*}
		\widehat{U}=
		\underset{U \in \mathbb{R}^{3d \times k}, U^{\mathrm{T}} U=I_k}{\operatorname{argmax}}
		\operatorname{Tr}\left[\left(U^T B^T B U\right)^{-1}\left(U^T A^T A U\right)\right].
	\end{equation*}
	In this experiment, we set $m=2500$, $d=200$.
        \begin{figure}[t!]
		\centering
		\subfigure
		{\includegraphics[width=0.46\textwidth,height=0.3\textwidth]{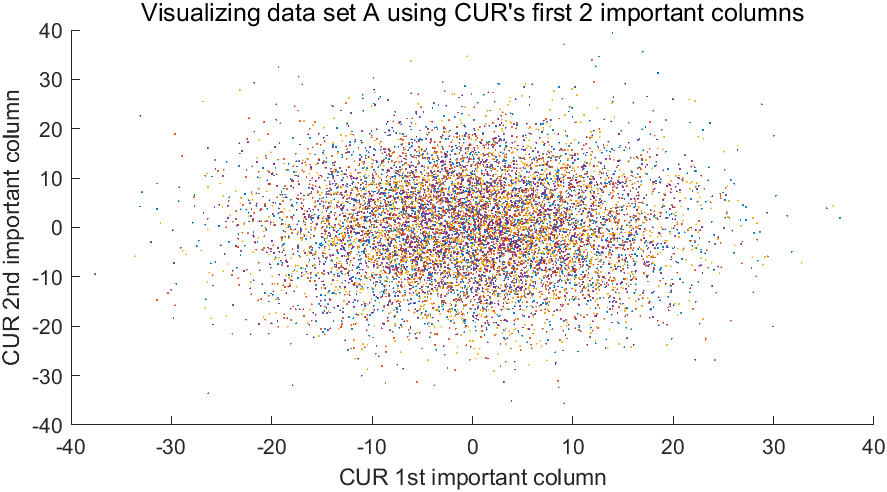}	}
		\subfigure
		{\includegraphics[width=0.46\textwidth,height=0.3\textwidth]{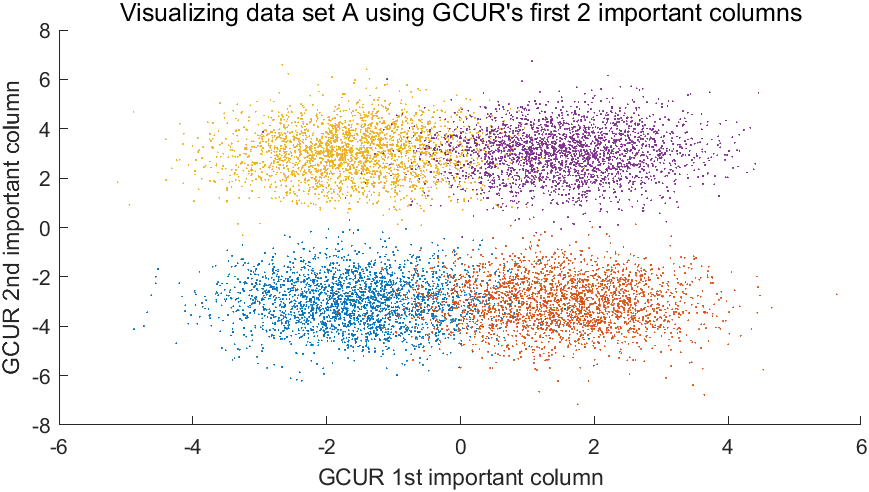}	}
		\subfigure
		{\includegraphics[width=0.46\textwidth,height=0.3\textwidth]{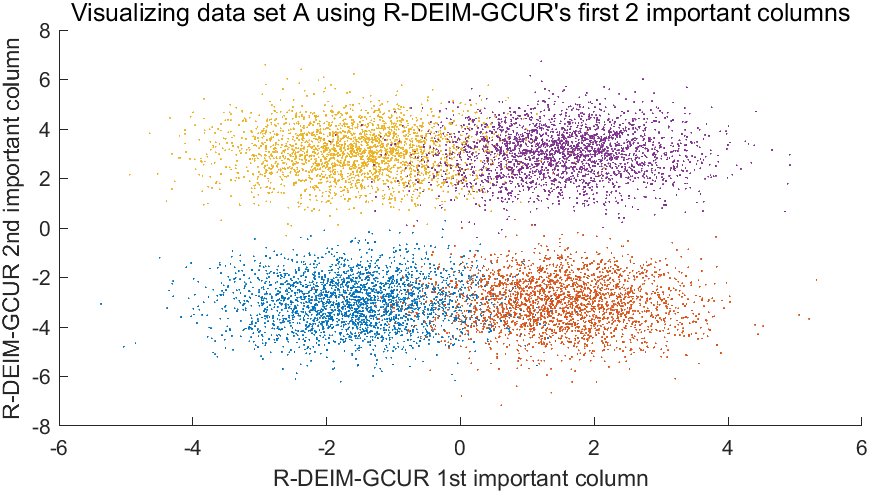}	}
		\subfigure
		{\includegraphics[width=0.46\textwidth,height=0.3\textwidth]{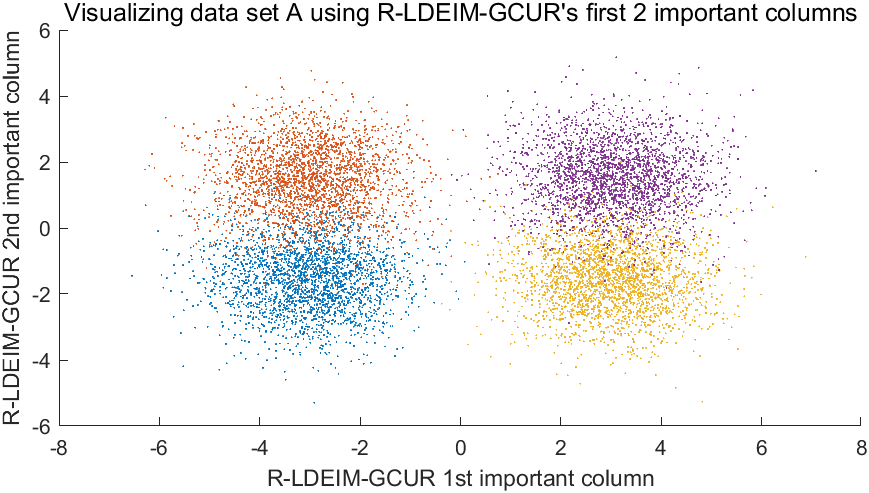}	}
		\caption{
			In the top figure, we visualize the data using the first two columns selected by CUR (left) and GCUR (right), respectively.
			In the bottom figure, we visualize the data using the first two columns selected by the R-DEIM-GCUR (left) and R-LDEIM-GCUR (right).
			In this experiment, we set $m=2500$, $d=200$ and $k=30$.
			The input oversampling parameters for the R-DEIM-GCUR and R-LDEIM-GCUR are set to be $70$ and $100$, respectively, and $\widehat{k}=k/2$.
				}
		\label{Fig1-Exp2}
	    \end{figure}
	Figure \ref{Fig1-Exp2} 
	is a visualization
	of the data using the first two important columns selected using the algorithms DEIM-CUR, GCUR, R-DEIM-GCUR and R-LDEM-GCUR for two of input matrix dimensions, respectively.
	It can be seen that the GCUR and R-GCUR methods produce a much clearer subgroup
	separation than the CUR.
	To a large extent, the GCUR and R-GCUR are able to differentiate the subgroups, while the CUR fails to do so.
	In terms of the running time, 
	the nonrandom GCUR costs $1.4250$ seconds while the R-DEIM-GCUR and R-LDEIM-GCUR spend $0.83059$ seconds and $0.82679$ seconds.
	
	\hskip 2em
	For further investigation, we emulate the manipulations in \cite{gidisu2022SIAMgeneralized} by comparing the performance of three subset selection methods: DEIM-CUR on $A$, GCUR, and R-GCUR on the matrix pair $(A, B)$, in identifying the subgroups of $A$.
	We accomplish this by selecting a subset of the columns of $A$ and comparing the classification results of each method.
	To evaluate the effectiveness of each method, we perform a 10-fold cross-validation \cite[p. 181]{james2013introduction} and apply the ECOC (Error-Correcting Output Codes) [13] and classification tree \cite{banfield2006IEEEcomparison} as the classifiers on the reduced data set, using the functions $\mathtt{fitcecoc}$ and $\mathtt{fitctree}$ with default parameters in MATLAB.
	Our results, presented in Table \ref{Table1-Exp2}, demonstrate that the R-LDEIM-GCUR method achieves the lowest classification error rate using both the ECOC and tree classifiers, while the standard DEIM-CUR method performs the worst.

	\begin{table}[htbp]
		\caption{k-Fold loss is the average classification loss overall 10-fold using CUR, GCUR, and R-GCUR as dimension reduction.
		The second and third columns give dimension information $m_1=2500$, $d_1=200$, $m_2=3000$, $d_2=200$,
		and the information on the number of columns $k=30$, selected from the data set using GCUR and R-GCUR for the ECOC classifier, likewise for the fifth and sixth columns for the tree classifier.}
		\label{Table1-Exp2}
		\setlength{\tabcolsep}{0.5mm}
		\centering
	\begin{tabular}{llllll}
		\hline Method & \multicolumn{2}{c}{$k$-Fold Loss } & Method & \multicolumn{2}{c}{$k$-Fold Loss } \\
		& \multicolumn{1}{c}{($m_1,d_1$)} & ($m_2,d_2$) & & \multicolumn{1}{c}{($m_1,d_1$)} & ($m_2,d_2$) \\
		\hline
		CUR+ECOC & $0.7512$ & $0.7521$ & CUR+Tree & $0.7488$ & $0.7465$ \\
		GCUR+ECOC & $0.0669$ & $0.0666$ & GCUR+Tree & $0.0986$ & $0.09758$ \\
		R-DEIM-GCUR+ECOC & $0.06930$ & $0.06700$ & R-DEIM-GCUR+Tree & $0.1000$ & $0.09558$ \\
		R-LDEIM-GCUR+ECOC & $0.06680$ & $0.06358$ & R-LDEIM-GCUR+Tree & $0.0980$ & $0.09691$ \\
		\hline
	\end{tabular}
	\end{table}
	$\mathbf{Example}$ $\mathbf{5.3}$
	This experiment is adapted from \cite[Experiment 5.4]{gidisu2022SIAMgeneralized}.
	In this study, we evaluate the performance of R-GCUR, GCUR, and CUR on public data sets.
	Specifically, we analyze single-cell RNA expression levels of bone marrow mononuclear cells (BMMCs) obtained from an acute myeloid leukemia (AML) patient and two healthy individuals.
	The data sets are processed as in \cite{boileau2020BIOexploring}, retaining the 1000 most variable genes across all 16856 cells,
	which includes 4501 cells of the patient-035, 1985 cells, and the other of 2472 cells of the two healthy individuals.
	By the operation in \cite{gidisu2022SIAMgeneralized}, we construct the $4457 \times 1000$ sparse background data matrix from the two healthy patients, containing $1,496,229$ non-zeros entries, and the target data matrix from patient-035, which has $1,628,174$ non-zeros entries.
	Our objective is to investigate the ability of CUR, GCUR, and R-GCUR to capture the biologically meaningful information related to the AML patient's BMMC cells pre- and post-transplant.
	As depicted in Figure \ref{Fig1-Exp3}, the GCUR and R-GCUR produce nearly linearly separable clusters that correspond to pre- and post-treatment cells.
	The R-DEIM-GCUR and R-LDEIM-GCUR outperform GCUR in terms of running time due to their significantly lower computational cost.
	Notably, the running time of the nonrandom GCUR algorithm is roughly twice that of our randomized algorithms.
	Importantly, all methods successfully capture biologically meaningful information, effectively separating the pre- and post-transplant cell samples.
	\begin{figure}[t!]
		\centering
		\subfigure
		{\includegraphics[width=0.46\textwidth,height=0.3\textwidth]{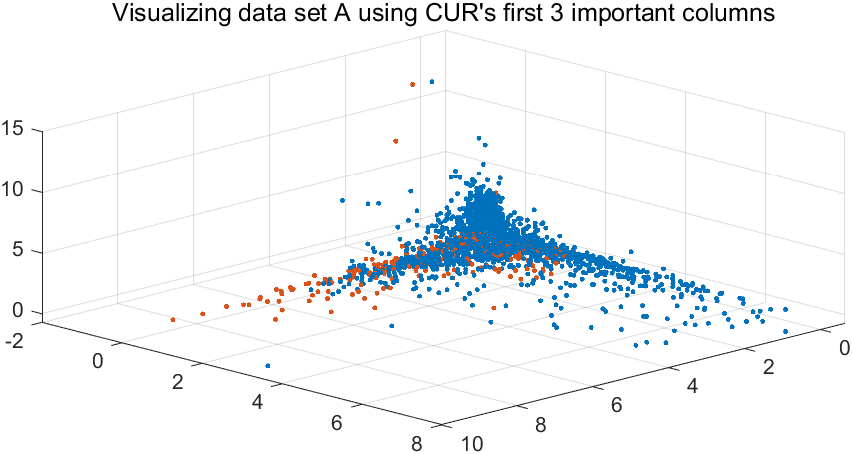}	}
		\subfigure
		{\includegraphics[width=0.46\textwidth,height=0.3\textwidth]{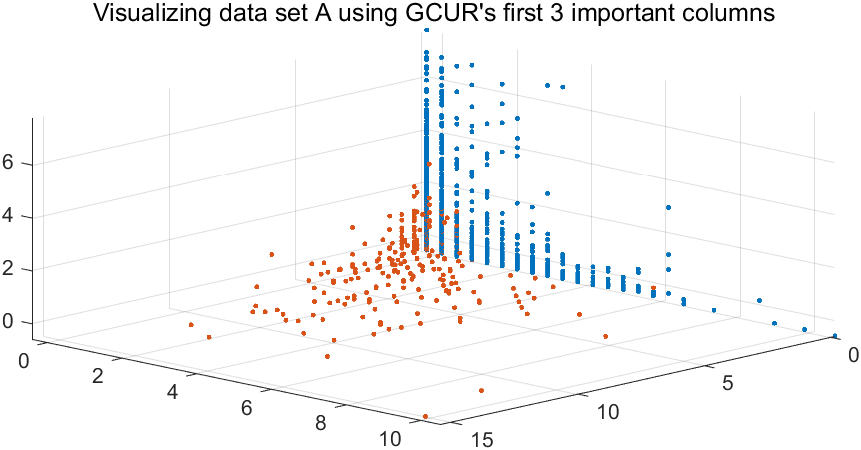}	}
		\subfigure
		{\includegraphics[width=0.46\textwidth,height=0.3\textwidth]{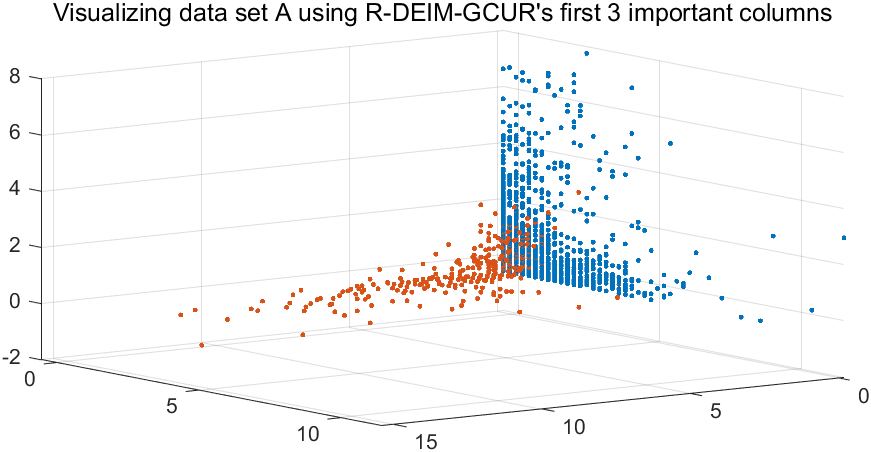}	}
		\subfigure
		{\includegraphics[width=0.46\textwidth,height=0.3\textwidth]{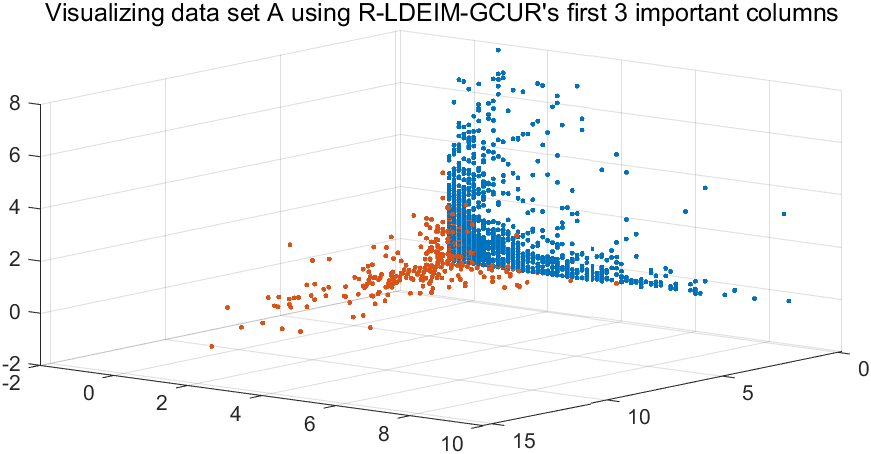}	}
		\caption{
			Acute myeloid leukemia patient-035 scRNA-seq data.
			In the top figure, we visualize the data using
			the first three genes selected by DEIM-CUR (top-left) and GCUR (top-right), and the CUR does not effectively give a discernible cluster of the pre- and post transplant cells.
			In the bottom figure, we visualize the data using the first three genes selected by R-DEIM-GCUR (bottom-left) and R-LDEIM-GCUR (bottom-right), which both produce almost linearly separable clusters which correspond to pre- and posttreatment cells.
		}
		\label{Fig1-Exp3}
	\end{figure}

	$\mathbf{Example}$ $\mathbf{5.4}$
	This experiment demonstrates the performance of our randomized algorithm for producing the RSVD-CUR decomposition.
	This test is an adaptation of \cite[Experiment 1]{gidisu2022ARXIVrsvd}, which considers a matrix perturbation problem of the form
		$A_E = A + B F G$,
	where $A\in\mathbb{R}^{m \times n}$,
	matrices $B\in\mathbb{R}^{m \times l}$, $G\in\mathbb{R}^{d \times n}$
	are noises distributed normally with mean $0$ and unit variance,
	and our goal is to reconstruct a low-rank matrix $A$ from $A_E$.
	We evaluate and compare a rank-$k$ RSVD-CUR decomposition of $A_E$, obtained by the nonrandom RSVD-CUR algorithm and its counterpart randomized algorithm, in terms of reconstructing matrix $A$ and the running time.
	The approximation quality of the decomposition is assessed by the relative
	matrix approximation error, i.e.,
	$\mathrm{Err}=\|A-\widetilde{A}\| /\|A\|$,
	where $\widetilde{A}$ is the reconstructed low-rank matrix.
	As an adaptation of the experiment in
	\cite[Example 1]{sorensen2016SIAMdeim} and \cite[Experiment 1]{gidisu2022ARXIVrsvd},
	we generate a rank-$100$ sparse nonnegative matrix
	$A\in \mathbb{R}^{m\times n}$ of the form

	\begin{equation*}
		A=\sum_{j=1}^{10} \frac{2}{j} \mathbf{x}_j \mathbf{y}_j^{\mathrm{T}}+\sum_{j=11}^{100} \frac{1}{j} \mathbf{x}_j \mathbf{y}_j^{\mathrm{T}}
	\end{equation*}
	where $\mathbf{x}_j\in\mathbb{R}^{m}$ and $\mathbf{x}_j\in\mathbb{R}^{n}$
	are random sparse vectors with nonnegative entries.
	We then perturb $A$ with a nonwhite noise matrix $BFG$ \cite{hansen1998SIAMrank}.
	The resulting perturbed	matrix we use is of the form
	\begin{equation*}
		A_E=A+\varepsilon \frac{\|A\|}{\|B F G\|} B F G,
	\end{equation*}
	where $\varepsilon$ is the noise level.
	Given each noise level $\varepsilon\in\{0.1, 0.15, 0.2\}$, we generate the RSVD-CUR decomposition computed by the RSVD-CUR algorithms and the randomized algorithm for varying dimensions and the target rank $k$ values.
	Here we set the parameter $\widehat{k}$, contained in the L-DEIM to be $\widehat{k} = k/2$ and $\widehat{k}=k$, respectively.
	The corresponding results are displayed in Tables \ref{Table1-Exp4}, \ref{Table2-Exp4} and \ref{Table3-Exp4}, where we can see that the randomized algorithms give comparable relative errors at substantially less cost.
	It indicates that using the random sampling techniques and L-DEIM
	method leads to a dramatic speed-up over classical approaches.
	\begin{table}[t!]
		\caption{Comparison of RSVD-CUR and randomized algorithms in CPU and relative error as the dimension $l$, $d$, $m$, $n$ ( we set $m=n$) and the target rank $k$ increase, with noise level $\varepsilon=0.1$.  }
		\label{Table1-Exp4}
		\setlength{\tabcolsep}{0.1mm}
		\centering
		\begin{tabular}{|ccc|c|c|c|}
			\hline
			\multicolumn{3}{|c|}{$(l,d,m,k)$}
			& $(1000,500,100,10)$  & $(5000,1000,100,20)$  & $(7000,2000,200,30)$  \\ \hline
			\multicolumn{2}{|c|}{\multirow{2}{*}{DEIM-RSVD-CUR}}
			& Err & $0.095573$  & $0.085198$  & $0.084425$  \\ \cline{3-6}
			\multicolumn{2}{|c|}{}
			& CPU & $0.099726$  & $8.1768$ & $15.251$ \\ \hline
			\multicolumn{2}{|c|}{\multirow{2}{*}{LDEIM-RSVD-CUR}}
			& Err & $0.11652$ & $0.094442$ & $0.083729$ \\ \cline{3-6}
			\multicolumn{2}{|c|}{}
			& CPU & $0.098987$ & $8.5575$ & $15.886$ \\ \hline
			\multicolumn{3}{|c|}{oversampling parameter}
			& $80$  & $500$  & $500$  \\ \hline
			\multicolumn{1}{|c|}{\multirow{4}{*}{R-LDEIM-RSVD-CUR}}
			& \multicolumn{1}{c|}{\multirow{2}{*}{$k=\widehat{k}$}}
			& Err & $0.095573$ & $0.085198$ & $0.084425$ \\ \cline{3-6}
			\multicolumn{1}{|c|}{} & \multicolumn{1}{c|}{}
			& CPU & $0.028330$ & $0.057914$ & $0.39295$ \\ \cline{2-6}
			\multicolumn{1}{|c|}{}
			& \multicolumn{1}{c|}{\multirow{2}{*}{$\widehat{k}=k/2$}}
			& Err & $0.095573$ & $0.085198$ & $0.084425$ \\ \cline{3-6}
			\multicolumn{1}{|c|}{} & \multicolumn{1}{c|}{}
			& CPU & $0.024517$ & $0.056423$ & $0.50397$ \\ \hline
		\end{tabular}

	\end{table}

	\begin{table}[t!]
		\caption{Comparison of RSVD-CUR and randomized algorithms in CPU and relative error as the dimension $l$, $d$, $m$, $n$ ( we set $m=n$) and the target rank $k$ increase, with noise level $\varepsilon=0.15$.  }
		\label{Table2-Exp4}
		\setlength{\tabcolsep}{0.1mm}
		\centering
			\begin{tabular}{|ccc|c|c|c|}
				\hline
				\multicolumn{3}{|c|}{$(l,d,m,k)$}
				& $(5000,1000,200,20)$  & $(10000,2000,500,30)$  & $(20000,2000,500,40)$  \\ \hline
				\multicolumn{2}{|c|}{\multirow{2}{*}{DEIM-RSVD-CUR}}
				& Err & $0.13123$  & $0.15709$  & $0.14705$  \\ \cline{3-6}
				\multicolumn{2}{|c|}{}
				& CPU & $7.5313$  & $62.594$ & $328.99$ \\ \hline
				\multicolumn{2}{|c|}{\multirow{2}{*}{LDEIM-RSVD-CUR}}
				& Err & $0.13103$ & $0.16492$ & $0.15604$ \\ \cline{3-6}
				\multicolumn{2}{|c|}{}
				& CPU & $7.3077$ & $61.396$ & $330.20$ \\ \hline
				\multicolumn{3}{|c|}{oversampling parameter}
				& $500$  & $500$  & $500$  \\ \hline
				\multicolumn{1}{|c|}{\multirow{4}{*}{R-LDEIM-RSVD-CUR}}
				& \multicolumn{1}{c|}{\multirow{2}{*}{$k=\widehat{k}$}}
				& Err & $0.13123$ & $0.15709$ & $0.14705$ \\ \cline{3-6}
				\multicolumn{1}{|c|}{} & \multicolumn{1}{c|}{}
				& CPU & $0.33164$ & $3.0591$ & $2.7742$ \\ \cline{2-6}
				\multicolumn{1}{|c|}{}
				& \multicolumn{1}{c|}{\multirow{2}{*}{$\widehat{k}=k/2$}}
				& Err & $0.13123$ & $0.15709$ & $0.14705$ \\ \cline{3-6}
				\multicolumn{1}{|c|}{} & \multicolumn{1}{c|}{}
				& CPU & $0.34972$ & $2.5485$ & $2.9289$ \\ \hline
			\end{tabular}
	\end{table}
	
	\begin{table}[t!]
		\caption{Comparison of RSVD-CUR and randomized algorithms in CPU and relative error as the dimension $l$, $d$, $m$, $n$ ( we set $m=n$) and the target rank $k$ increase, with noise level $\varepsilon=0.2$.  }
		\label{Table3-Exp4}
		\setlength{\tabcolsep}{0.1mm}
		\centering
		\begin{tabular}{|ccc|c|c|c|}
			\hline
			\multicolumn{3}{|c|}{$(l,d,m,k)$}
			& $(5000,1000,100,10)$  & $(10000,1000,500,30)$  & $(7000,2000,200,30)$  \\ \hline
			\multicolumn{2}{|c|}{\multirow{2}{*}{DEIM-RSVD-CUR}}
			& Err & $0.15325$  & $0.18345$  & $0.18429$  \\ \cline{3-6}
			\multicolumn{2}{|c|}{}
			& CPU & $7.9139$  & $56.001$ & $384.17$ \\ \hline
			\multicolumn{2}{|c|}{\multirow{2}{*}{LDEIM-RSVD-CUR}}
			& Err & $0.14943$ & $0.21517$ & $0.19300$ \\ \cline{3-6}
			\multicolumn{2}{|c|}{}
			& CPU & $7.9627$ & $51.571$ & $357.91$ \\ \hline
			\multicolumn{3}{|c|}{oversampling parameter}
			& $100$  & $500$  & $500$  \\ \hline
			\multicolumn{1}{|c|}{\multirow{4}{*}{R-LDEIM-RSVD-CUR}}
			& \multicolumn{1}{c|}{\multirow{2}{*}{$k=\widehat{k}$}}
			& Err & $0.15325$ & $0.18345$ & $0.18429$ \\ \cline{3-6}
			\multicolumn{1}{|c|}{} & \multicolumn{1}{c|}{}
			& CPU & $0.079395$ & $2.1681$ & $3.9116$ \\ \cline{2-6}
			\multicolumn{1}{|c|}{}
			& \multicolumn{1}{c|}{\multirow{2}{*}{$\widehat{k}=k/2$}}
			& Err & $0.15325$ & $0.18345$ & $0.18429$ \\ \cline{3-6}
			\multicolumn{1}{|c|}{} & \multicolumn{1}{c|}{}
			& CPU & $0.06830$ & $2.0079$ & $3.8923$ \\ \hline
		\end{tabular}
	\end{table}
	\section{Conclusion}
	\hskip 2em
	In this paper, by combining the random sampling techniques with the L-DEIM method, we
	develop new efficient randomized algorithms for computing the GCUR decomposition for matrix pairs and the RSVD-CUR decomposition for matrix triplets with a given target rank.
	We also provided the detailed probabilistic analysis for the proposed randomized algorithms.
	Theoretical analyses and numerical examples illustrate that exploiting the randomized techniques results in a significant improvement in terms of the CPU time while keeping a high degree of accuracy.
	Finally, it is natural to consider applying the L-DEIM for developing randomized algorithms that adaptively find a low-rank representation satisfying a given tolerance, which is beneficial when the target rank is not known in advance, and it will be discussed in our future work.
	

  \section*{Funding}
  Z. Cao is supported by the National Natural Science Foundation of China under Grant 11801534.
  Y. Wei is supported by the National Natural Science Foundation of China under Grant 12271108 and the Innovation Program of Shanghai Municipal Education Committee.
  P. Xie is supported by the National Natural Science Foundation of China under Grants 12271108, 11801534 and the Fundamental Research Funds for the Central Universities under Grant 202264006.

    \section*{Declarations}
  The authors have not disclosed any competing interests.

    \section*{Data  Availability  Statements}
  All datasets are publicly available.

	\bibliographystyle{siam}
	\bibliography{RGCUR}	
\end{document}